\documentclass[11pt,twoside,leqno]{article}
\usepackage{amsfonts,amssymb,amscd,amsmath,amsthm,enumerate,verbatim,calc,latexsym}
\usepackage[colorlinks=true,linkcolor=blue,urlcolor=blue,citecolor=blue,anchorcolor=blue,bookmarksnumbered=blue]{hyperref}
\usepackage[all]{xy}
\usepackage[dvipsnames]{xcolor}
\newtheorem{theorem}{Theorem}[section]
\newtheorem{proposition}[theorem]{Proposition}
\newtheorem{corollary}[theorem]{Corollary}
\newtheorem{lemma}[theorem]{Lemma}
\newtheorem{definition}[theorem]{Definition}

\newtheorem{example}[theorem]{Example}
\newtheorem{remark}[theorem]{Remark}

\numberwithin{equation}{section}
\begin{document}
\title{\bf  $(n,d)$-injective and  $(n,d)$-flat modules under a special semidualizing bimodule}
\author{
\\{\bf Mostafa Amini}\\
\small Department of Mathematics, Payame Noor University, Tehran, Iran\\
\small E-mail: dr.mostafa56@pnu.ac.ir\\
\\{\bf Alireza Vahidi}\\
\small Department of Mathematics, Payame Noor University, Tehran, Iran\\
\small E-mail: vahidi.ar@pnu.ac.ir\\
\\{\bf Fatemeh Ghanavati}\\
\small Department of Mathematics, Payame Noor University, Tehran, Iran\\
\small E-mail: fatemehghanavati@student.pnu.ac.ir
}

\date{}
\maketitle

\begin{abstract}
Let $S$ and $R$ be rings, $n, d\geq 0$ be two  integers or $n=\infty$.
 In this paper, first we introduce   special (faithfully)  semidualizing bimodule $_S(K_{d-1})_R$, and then introduce and study the concepts of $K_{d-1}$-$(n,d)$-injective (resp. $K_{d-1}$-$(n,d)$-flat) modules as a common generalization of some known modules such as
$C$-injective, $C$-weak injective and $C$-$FP_n$-injective  (resp. $C$-flat, $C$-weak flat and $C$-$FP_n$-flat) modules.  Then  we obtain some
characterizations of  two classes of
these modules, namely $\mathcal{I}_{K_{d-1}}^{(n,d)}(R)$ and $\mathcal{F}_{K_{d-1}}^{(n,d)}(S)$. We show that the cleasses $\mathcal{I}_{K_{d-1}}^{(n,d)}(R)$ and $\mathcal{F}_{K_{d-1}}^{(n,d)}(S)$ are covering and preenveloping.
Also, we investigate Foxby equivalence relative to the classes of this modules. Finally over $n$-coherent rings, we prove that  the classes $\mathcal{I}_{ K_{d-1}}^{(n,d)}(R)_{<\infty}$  and $\mathcal{F}_{ K_{d-1}}^{(n,d)}(S)_{<\infty}$ are closed under extentions, kernels of epimorphisms and cokernels of monomorphisms. \\
\vspace*{-0.25cm}
\end{abstract}
\section{Introduction}\label{1}
Injectivity and flatness of modules under a semidualizing module has become an important and active area of research in homological algebra, where  over a commutative Noetherian ring $R$,   a semidualizing module for $R$ is a finite (that is finitely generated) $R$-module
$C$ with ${\rm Hom}_{R}(C,C)$ is canonically isomorphic to $R$ and ${\rm Ext}_{R}^{i}(C,C)=0$ for any $i\geq1$. Semidualizing modules (under
different names) were independently studied by Foxby in \cite{H.B}, Golod in \cite{E.SG} and Vasconcelos in \cite{WWV}.
In 2005,  Araya, Takahashi and Yoshino in \cite{AR} extended the notion of semidualizing
modules to a pair of non-commutative, but Noetherian rings. Also in 2007,
Holm and White in \cite{HW}, generalized the notion of a semidualizing module to
general associative rings, and defined and studied Auslander and Bass classes
 under a semidualizing bimodule $C$. Then, using semidualizing bimodule $C$, they introduced notions of $C$-injective, $C$-projective and $C$-flat modules.
 
 In 2017, Gao and Zhao in \cite{Z.TT} introduced the concept of $C$-weak injective (resp. $C$-weak flat)
 modules with respect to semidualizing bimodule $C$ as a generalization of $C$-injective (resp.  $C$-flat) modules, where weak injective and weak flat
 modules  were already introduced by Gao and Wang \cite{Z.G}. They showed that the Auslander and Bass classes  contain all weak injective and weak flat modules, respectively, and then they investigated Foxby equivalence relative to  the classes of this modules. In 2022, Wu and Gao in \cite{WGZ},  introduced the notion of $C$-$FP_n$-injective (resp. $C$-$FP_n$-flat) modules as a common generalization of some known modules such as $C$-injective, and $C$-$FP$-injective and $C$-weak injective (resp. $C$-flat, $C$-weak flat) modules. Furthermore, they  proved that the classes of this modules are  preenveloping and covering, and found that when these classes are closed under extensions, cokernels of monomorphisms, and kernels of epimorphisms.

Let $R$ and $S$ be rings, and let $n,d$ be  non-negative integers.  In this paper, first we introduce the concept of a special semidualizing bimodule $_S(K_{d-1})_{R}$, where $K_{d-1}$ is  the $(d-1)$th syzygy  of a super finitely presented $_SC_R$, that is semidualizing.  Then we  study the relative homological algebra
associated to the notions of $(n,d)$-injective and $(n,d)$-flat modules with respect to a special semidualizing
bimodule $_S(K_{d-1})_{R}$, where $(n,d)$-injective and $(n,d)$-flat
 modules  were already introduced by Zhou \cite{TX}. We show that $K_{d-1}$-$(n,d)$-injective (resp. $K_{d-1}$-$(n,d)$-flat)  modules possess many nice properties
analogous to that of $C$-weak injective  (resp. $C$-$FP_{n}$-injective) and $C$-weak flat (resp. $C$-$FP_{n}$-flat)  modules as in \cite{Z.TT,TX}.
This paper is organized as follows:\\
\textbf{In Sec. 2}, some fundamental notions and some preliminary results are stated.\\
\textbf{In Sec. 3}, we introduce $K_{d-1}$-$(n,d)$-injective  and $K_{d-1}$-$(n,d)$-flat modules, where $K_{d-1}$ is a special semidualizing
bimodule. For any  $n^{'}\geq n$ and $d^{'}\geq d$,  every  $K_{d-1}$-$(n,d)$-injective (resp. $K_{d-1}$-$(n,d)$-flat)   module is  $K_{d-1}$-$(n^{'},d)$-injective (resp. $K_{d-1}$-$(n^{'},d)$-flat), but not conversely, and also, over $n$-coherent rings, every  $K_{d-1}$-$(n,d)$-injective (resp. $K_{d-1}$-$(n,d)$-flat)   module is  $K_{d-1}$-$(n^{'},d^{'})$-injective (resp. $K_{d-1}$-$(n^{'},d^{'})$-flat), but not conversely, see Example \ref{2-1-2}.  Then for $n>d+1$ with that $d\geq 1$, we prove that $\mathcal{I}^{(n,d)}(S)\subseteq\mathcal{B}_{ K_{d-1}}(S)$ and $\mathcal{F}^{(n,d)}(R)\subseteq\mathcal{A}_{ K_{d-1}}(R)$, where $\mathcal{I}^{(n,d)}(S)$, $\mathcal{F}^{(n,d)}(R)$, $\mathcal{B}_{ K_{d-1}}(S)$ and $\mathcal{A}_{ K_{d-1}}(R)$ denote class of all $(n,d)$-injective $S$-modules,  class of all $(n,d)$-flat $R$-modules, Bass class and Auslander class, respectively. Also, we show that the classes $\mathcal{I}_{K_{d-1}}^{(n,d)}(R)$ and $\mathcal{F}_{K_{d-1}}^{(n,d)}(S)$ are
closed under extentions, direct summands, direct products, direct sums, pure submodules and pure quotients,  where $\mathcal{I}_{K_{d-1}}^{(n,d)}(R)$ and $\mathcal{F}_{K_{d-1}}^{(n,d)}(S)$ denote class of all $K_{d-1}$-$(n,d)$-injective $R$-modules,  class of all $K_{d-1}$-$(n,d)$-flat $R$-modules, respectively. Moreover,
we deduce that   classes  $\mathcal{I}_{ K_{d-1}}^{(n,d)}(R)$ and  $\mathcal{F}_{ K_{d-1}}^{(n,d)}(S)$ are  covering and preenveloping.\\
\textbf{In Sec. 4}, by considering special faithfully semidualizing bimodule $K_{d-1}$, we provide additional information concerning the
Foxby equivalence between the subclasses of Auslander class $\mathcal{A}_{ K_{d-1}}(R)$ and that of the Bass class
$\mathcal{B}_{ K_{d-1}}(S)$. Then over $n$-coherent rings, we show that the classes $\mathcal{I}^{(n,d)}_{K_{d-1}}(R)_{<\infty}$ and $\mathcal{F}^{(n,d)}_{K_{d-1}}(S)_{<\infty}$ are closed under extentions, kernels of epimorphisms and cokernels of monomorphisms.\\

\section{Preliminaries}\label{2}



\ \ Let $n,d$ be  non-negative integers. Throughout this paper $R$ and $S$ are fixed associative rings with
unities and all $R$-or $S$-modules are understood to be unital left $R$-or $S$-modules (unless specified otherwise).
Right $R$-or $S$-modules are identified with left modules over the opposite rings $R^{op}$
or $S^{op}$. $_SM_R$ is used to denote that $M$ is an $(S, R)$-bimodule. This means that $M$
is both a left $S$-module and a right $R$-module, and these structures are compatible.

In this section, some fundamental concepts and notations are stated.
\begin{definition}{\rm (\cite{ARB,S-Q})}\label{1.lk-1}
 \begin{enumerate}
\item [\rm (1)]
{\rm  An $R$-module $U$ is called {\it  finitely $n$-presented} if there exists an exact sequence
$$F_{n}\longrightarrow F_{n-1}\longrightarrow\cdots\longrightarrow  F_1\longrightarrow  F_0\longrightarrow  U\longrightarrow  0,$$ where each $F_i$
is finitely generated and free;
\item [\rm (2)]
An $R$-module $M$ is called  {\it $FP_{n}$-injective} if ${\rm Ext}_{R}^{1}(U, M)=0$ for any  finitely $n$-presented $R$-module $U$, and a right $R$-module $M$ is called  {\it $FP_{n}$-flat} if ${\rm Tor}_{1}^{R}(M, U)=0$ for any  finitely $n$-presented  $R$-module $U$;
\item [\rm (3)]
A ring $R$ is called left \textit{$n$-coherent} if every finitely $n$-presented $R$-module is finitely $(n+ 1)$-presented.}
 \end{enumerate}
\end{definition}
\begin{definition}{\rm (\cite{Z.G,Z.W})}\label{1.lk-2}
 \begin{enumerate}
\item [\rm (1)]
{\rm  An $R$-module $U$ is called {\it super finitely presented}  if there exists an exact sequence
$$ \cdots\longrightarrow  F_{2}\longrightarrow  F_1\longrightarrow  F_0\longrightarrow  U\longrightarrow  0,$$ where each $F_i$
is finitely generated and free;
\item [\rm (2)]
A module $M$ is called  {\it weak injective or $FP_{\infty}$-injective} if ${\rm Ext}_{R}^{1}(U, M)=0$ for any super finitely presented $R$-module $U$, and a right $R$-module $M$ is called  {\it weak flat or $FP_{\infty}$-flat} if ${\rm Tor}_{1}^{R}(M, U)=0$ for any super finitely presented  $R$-module $U$.}
 \end{enumerate}
\end{definition}

\begin{definition}{\rm (\cite{AR,HW})}\label{1.lk-3}
{\rm  Let $R$ and $S$ be rings.} 
\begin{enumerate}
\item [\rm (1)]
{\rm An $(S, R)$-bimodule $C = _S C_R$ is {\it semidualizing} if the following conditions are satisfied:\\
$(a_1)$  $_SC$ admits a degreewise finite $S$-projective resolution;\\
$(a_2)$  $C_R$ admits a degreewise finite $R^{op}$-projective resolution;\\
$(b_1)$  The homothety map $_S\gamma: _SS_S\rightarrow {\rm Hom}_{R^{op}}(C,C)$ is an isomorphism;\\ 
$(b_2)$  The homothety map $\gamma: _RR_R\rightarrow {\rm Hom}_{S}(C,C)$ is an isomorphism; \\ 
$(c_1)$ ${\rm Ext}_{S}^{i}(C,C)=0$ for all $i\geq1$;\\
$(c_2)$ ${\rm Ext}_{R^{op}}^{i}(C,C)=0$ for all $i\geq1$.
\item [\rm (2)]
A semidualizing bimodule $_S C_R$ is faithfully semidualizing if it satisfies the
following conditions for all modules $_S N$ and $M_R$:\\
(a) If ${\rm Hom}_{S}(C,N)=0$, then $N=0$;\\
(b) If ${\rm Hom}_{R^{op}}(C,M)=0$, then $M=0$.}
 \end{enumerate}
\end{definition}
\begin{definition}{\rm (\cite{HW})}\label{1.lk-4}
 \begin{enumerate}
\item [\rm (1)]
{\rm  The Auslander class $\mathcal{A}_{C}(R)$ with respect to $C$ consists of all modules $M$ in ${\rm Mod} R$ satisfying:\\
$(A_1)$  ${\rm Tor}_{i}^{R}(C,M)=0$ for all $i\geq 1$.\\
$(A_2)$  ${\rm Ext}_{S}^{i}(C,C\otimes_{R}M)=0$ for all $i\geq 1$.\\
$(A_3)$  The natural evaluation homomorphism $\mu_{M}:M\rightarrow {\rm Hom}_{S}(C,C\otimes_{R}M)$ is an isomorphism (of left $R$-modules).
\item [\rm (2)]
The Bass class $\mathcal{B}_{C}(S)$ with respect to $C$ consists of all modules $N\in {\rm Mod} S$
satisfying:\\
$(B_1)$  ${\rm Ext}_{S}^{i}(C,N)=0$ for all $i\geq 1$.\\
$(B_2) $  ${\rm Tor}_{i}^{R}(C,{\rm Hom}_{S}(C,N))=0$ for all $i\geq 1$.\\
$(B_3)$  The natural evaluation homomorphism $\nu_{N}:C\otimes_{R}{\rm Hom}_{S}(C,N)\rightarrow N $ is an isomorphism (of left $S$-modules).}
\end{enumerate}
\end{definition}
\begin{definition}{\rm (\cite{Z.TT,WGZ})}\label{1.lk-5}
 \begin{enumerate}
\item [\rm (1)]
{\rm An $R$-module is called \textit{$C$-$FP_n$-injective} if it has the form $\operatorname{Hom}_{S}(C, I)$ for some $FP_n$-injective  $S$-module $I$. An $S$-module is called \textit{$C$-$FP_n$-flat} if it has the form $C\otimes_{R} F$ for some $FP_n$-flat $R$-module $F$;}
\item [\rm (2)]
{\rm An $R$-module is called \textit{$C$-weak injective} if it has the form $\operatorname{Hom}_{S}(C, I)$ for some weak injective  $S$-module $I$. An $S$-module is called \textit{$C$-weak flat} if it has the form $C\otimes_{R} F$ for some weak flat $R$-module $F$.}
\end{enumerate}
\end{definition}

\begin{definition}{\rm \cite[Definition 2.1]{TX}}\label{1.lk1} 
{\rm Let  $n, d$ be non-negative integers. An $S$-module $M$ is called
$(n,d)$-injective, if ${\rm Ext}_{S}^{d+1}(U,M)=0$ for every finitely $n$-presented  $S$-module $U$. Let
$n, d$ be non-negative integers and $n\geq 1$. An $R$-module $N$ is called
$(n,d)$-flat, if ${\rm Tor}_{R}^{d+1}(U,N)=0$ for every finitely $n$-presented  $R^{op}$-module $U$.}
\end{definition}
We denote by $\mathcal{I}^{(n,d)}(S)$ (resp. $\mathcal{F}^{(n,d)}(R)$) the class of all
$(n,d)$-injective $S$-modules (resp. the class of all $(n,d)$-flat $R$-modules).
\begin{remark}\label{1.lk2}
Let  $n, d$ be non-negative integers such that $n\geq d+1$, and $U$  a finitely $n$-presented $S$-module (resp. $R^{op}$-module). Then 
\begin{enumerate}
\item [\rm (1)]
There exists an exact sequence 
$$ F_{n}\to F_{n-1}\to\cdots\to F_1\to F_0\to U\to 0$$ of  $S$-modules (resp. $R^{op}$-modules), where each $F_i$ is  finitely generated and free for any $i\geq 0$.
 If $K:={\rm Ker}(F_{d-1}\to F_{d-2})$, then the module $K$ is called special  finitely presented. 
 \item [\rm (2)]
 Notice that ${\rm Ext}_{S}^{d+1}(U,-)\cong{\rm Ext}_{R}^{1}(K,-)$ and  ${\rm Tor}_{d+1}^{R}(U,-)\cong {\rm Tor}_{1}^{R}(K,-)$.
\end{enumerate}
\end{remark}
\begin{definition}\label{2.bbb}
{\rm Let  $n, d$ be non-negative integers such that $n\geq d+1$. Then 
the $(n,d)$-injective dimension of an $S$-module $M$ and $(n,d)$-flat dimension of an $R$-module $N$
are defined by}\\
 $(n,d)$.${\rm id}_{S}(M)$= ${\rm inf}\{k: \ {\rm Ext}_{R}^{d+k+1}(U, M)=0 \    {\rm for }\ {\rm every} \  {\rm finitely }\  n$-${\rm presented} \  S$-${\rm module} \ U\}$,
and
$(n,d)$.${\rm fd}_{R}(N)$= $ {\rm inf}\{k: \ {\rm Tor}^{R}_{d+k+1}(U, N)=0 \   {\rm for }\ {\rm every} \  {\rm finitely }\  n$-${\rm presented} \  R$-${\rm module} \  U\}$.
\end{definition}
--------------------------------------------------------------

\section{ $K_{d-1}$-$(n,d)$-injective and $K_{d-1}$-$(n,d)$-flat modules}
\ \ \ Let  $n,d$ be  non-negative integers. In this section, we introduce and study  $K_{d-1}$-$(n,d)$-injective and $K_{d-1}$-$(n,d)$-flat modules under a special semidualizing bimodule $_S(K_{d-1})_{R}$.  We start with the following definition. 
 \begin{definition}\label{1.1} 
{\rm Let  $d$ be a non-negative integer. A  super finitely presented $(S, R)$-bimodule $C =$ $_S C_R$ is said to be }{\it $d$-semidualizing} {\rm if the $(d-1)$th syzygy  $K_{d-1}$ of $C $ is semidualizing}.  {\rm \textbf{In this cas, we call that $K_{d-1}$ is a  special semidualizing  with respect to $C$.}}
\end{definition}
There are  examples of $d$-semidualizing bimodules, see  Example \ref{2-1-2}(1). 
\begin{definition}\label{1.a1} 
{\rm Let   $K_{d-1}$ be a special semidualizing bimodule  with respect to $C$, and $n,d\geq 0$.   An  $R$-module is called {\it $K_{d-1}$-$(n,d)$-injective} if it has the form
${\rm Hom}_{S}(K_{d-1}, I )$ for some  $I\in\mathcal{I}^{(n,d)}(S)$. An $S$-module  is called {\it $K_{d-1}$-$(n,d)$-flat} if it has the form $K_{d-1}\otimes_{R} F$ for some
$F\in\mathcal{F}^{(n,d)}(R)$.

We consider
$$\mathcal{I}_{K_{d-1}}^{(n,d)}(R)=\{{\rm Hom}_{S}(K_{d-1}, I )\mid I\in \mathcal{I}^{(n,d)}(S)\}$$ and
$$\mathcal{F}_{K_{d-1}}^{(n,d)}(R)=\{K_{d-1}\otimes_{R} F \mid F\in \mathcal{F}^{(n,d)}(R)\}.$$}
\end{definition}
\begin{remark}\label{1.pi-3}
 Let  $n,d$ be  non-negative integers. Then:
\begin{enumerate}
\item [\rm (1)]
Every $K_{d-1}$-$(n,d)$-injective  (resp. $K_{d-1}$-$(n,d)$-flat ) module is $K_{d-1}$-$(n^{'},d)$-injective (resp. $K_{d-1}$-$(n^{'},d)$-flat) for any $n^{'}\geq n$, but not conversely, since $(n^{'},d)$-injective  (resp. $(n^{'},d)$-flat ) modules need not be $(n,d)$-injective (resp. $(n,d)$-flat) for any $n^{'}> n$, (see Example \ref{2-1-2}(2));
\item [\rm (2)]
Let $K_{d-1}=K_{d^{'}-1}$. Then over $n$-coherent rings every $K_{d-1}$-$(n,d)$-injective  (resp. $K_{d-1}$-$(n,d)$-flat ) module is $K_{d-1}$-$(n^{'},d^{'})$-injective (resp. $K_{d-1}$-$(n^{'},d^{'})$-flat) for any $n^{'}\geq n$ and $d^{'}\geq d$, but not conversely, since $(n^{'},d^{'})$-injective  (resp. $(n^{'},d^{'})$-flat ) modules need not be $(n,d)$-injective (resp. $(n,d)$-flat),  see (Example \ref{2-1-2}(3)).
\item [\rm (2)]  ${\rm Ext}_{S}^{d+1}(C,-)\cong{\rm Ext}_{S}^{1}(K_{d-1},-)$ and  ${\rm Tor}_{d+1}^{R}(-,C)\cong {\rm Tor}_{1}^{R}(-,K_{d-1}).$
\end{enumerate}
\end{remark}
Recall that a ring $R$ is said to be an \textit{$(n, 0)$-ring} or \textit{$n$-regular ring} if every finitely $n$-presented $R$-module is projective (see \cite{L.g,.TZ}).
\begin{example}\label{2-1-2}
\begin{enumerate}
\item [\rm (1)]
If $R=S=C$, then $R$ is $d$-semidualizing bimodule;
\item [\rm (2)]
{\rm Let $K$ be a field, $E$ a $K$-vector space with infinite rank, and $A$ a Noetherian ring of global dimension $0$. Set $B= K\ltimes E$ the trivial extension of $K$ by $E$ and $R= A\times B$ the direct product of $A$ and $B$. By \cite[Theorem 3.4(3)]{L.g}, $R$ is a $(2, 0)$-ring which is not a $(1, 0)$-ring. Thus, for every $R$-module $M$ and every finitely $2$-presented $R$-module $L$, $\operatorname{Ext}_{R}^{1}(L, M)= 0$ (resp. $\operatorname{Tor}_{1}^{R}(L, M)= 0$) . Hence every $R$-module is $(2,0)$-injective (resp. $(2,0)$-flat). On the other hand, there exists an $R$-module which is not $(1,0)$-injective (resp. $(1,0)$-flat), since  if every $R$-module is $(1,0)$-injective (resp. $(1,0)$-flat), \cite[Theorem 3.9]{.TZ} implies that  $R$ is  $(1, 0)$-ring, contradiction. Also, since $C=R=S$ is $d$-semidualizing, then every $R$-module is $C$-$(2,0)$-injective and $C$-$(2,0)$-flat, and there exists an $R$-module which is not $C$-$(1,0)$-injective (resp. $C$-$(1,0)$-flat).
\item [\rm (3)]
Let $R$ be a  ring with $l$.$(1,0)$-${\rm dim}(R)\leq1$ but not $(1,0)$-ring,
for example, let $R = k[X]$ where $k$ is a field. Then there exists an  $R$-module
which is not $(1,0)$-injective by \cite[ Theorem 3.9]{.TZ}. We claim that every
 $R$-module is $(2,1)$-injective. Let $M$ be an $R$-module
and $U$ a $2$-presented  $R$-module. Then there exists an exact sequence $0\rightarrow  M\rightarrow E\rightarrow D\rightarrow 0$ with $E$ is injective. By \cite[Theorem 2.12]{.TZ}, $D$ is $(1,0)$-injective. From the exact sequence $0\rightarrow  {\rm Ext}_{R}^{1}(U,D)\rightarrow  {\rm Ext}_{R}^{1+1}(U,M)\rightarrow 0$ it follows that 
${\rm Ext}_{R}^{1+1}(U,M)=0$, and so every
 $R$-module $M$ is $(2,1)$-injective. Similarly, using from  \cite[Theorems 2.22 and 3.9]{.TZ},  every
 $R$-module is $(2,1)$-flat, but not $(1,0)$-flat. Let $C=R=S$. Since $R$ is $d$-semidualizing, we deduce that every $R$-module is $C$-$(2,1)$-injective (resp. $C$-$(2,1)$-flat), but not $C$-$(1,0)$-injective (resp. $C$-$(1,0)$-flat).
}
\end{enumerate}
\end{example}
\begin{remark}\label{1.pi3}
 Let  $n,d$ be  non-negative integers. Then:
\begin{enumerate}
\item [\rm (1)]
{\rm In case $d=0$, every $d$-semidualizing bimodule is semidualizing;
\item [\rm (2)]
 In case $d=0, n=0$ (resp. $d=0, n=1$),  $K_{d-1}$-$(n,d)$-injective $R$-modules are just the $C$-injective (resp. $C$-$FP$-injective) $R$-modules and $K_{d-1}$-$(n,d)$-flat $S$-modules are just the $C$-flat  $S$-modules in \cite{D.D,HM,WY,WZO};
\item [\rm (3)]
 In case $d=0$,  $K_{d-1}$-$(n,d)$-injective $R$-modules are just the $C$-$FP_n$-injective $R$-modules and $K_{d-1}$-$(n,d)$-flat $S$-modules are just the $C$-$FP_n$-flat   $S$-modules in \cite{WGZ};
 \item [\rm (5)]
 In case $d=0, n=\infty$,  $K_{d-1}$-$(n,d)$-injective $R$-modules are just the $C$-weak injective $R$-modules and $K_{d-1}$-$(n,d)$-flat $S$-modules are just the $C$-weak flat  $S$-modules in \cite{Z.TT};
 \item [\rm (6)] \textbf{In this paper,  $K_{d-1}$ be a special semidualizing bimodule, and we only focus on the case $n> d+1$ with $d\geq1$.}
 \item [\rm (7)] $\mathcal{B}_{K_{d-1}}(S)$ and $\mathcal{A}_{K_{d-1}}(R) $ are the Bass class and the Auslander class with respect to special semidualizing $K_{d-1}$, respectively. }
 \end{enumerate}
\end{remark}

\begin{lemma}\label{1.b-1}
The following assertions hold:
\begin{enumerate}
\item [\rm (1)] 
If $M$ is an $(n,d)$-injective $S$-module, then ${\rm Ext}_{S}^i(K_{d-1},M)=0$ for any $i\geq 1$;
 \item [\rm (2)] 
 If $N$ is an $(n,d)$-flat $R$-module, then ${\rm Tor}^{R}_i(K_{d-1},N)=0$ for any $i\geq 1$.
 \end{enumerate}
\end{lemma}
\begin{proof}
(1) Let $K_{d-1}$ be a special semidualizing with respect to super finitely presented bimodule $C$. Then $C$ has an infinite finite presentation
$$\cdots \longrightarrow F_d\stackrel{f_{d}}\longrightarrow F_{d-1}\longrightarrow\cdots \longrightarrow F_1\stackrel{f_1}\longrightarrow F_0\stackrel{f_0}\longrightarrow C\stackrel{f_{-1}}\longrightarrow 0.$$
Thus ${\rm Ext}_{S}^{d+j+1}(C,M)\cong {\rm Ext}_{S}^{d+1}({\rm Ker}f_{j-1},M)$ for any $j\geq 0$. Since $M$ is $(n,d)$-injective and ${\rm Ker}f_{j-1}$ is finitely $n$-presented, we have 
${\rm Ext}_{S}^{d+j+1}(C,M)\cong {\rm Ext}_{S}^{d+1}({\rm Ker}f_{j-1},M)=0.$ Also, we have ${\rm Ext}_{S}^{d+j+1}(C,M)\cong {\rm Ext}_{S}^{j+1}(K_{d-1},M).$ Hence ${\rm Ext}_{S}^{j+1}(K_{d-1},M)=0$ for any $j\geq 0$, and so ${\rm Ext}_{S}^i(K_{d-1},M)=0$ for any $i\geq 1$.

 (2) It is similar to the proof of (1).
\end{proof}

We denote the character module of $M$ by $M^{*}:= \operatorname{Hom}_{\mathbb{Z}}(M, {\mathbb{Q}}/{\mathbb { Z}})$ \cite[Page 135]{Rot2}.

When $R$ is a commutative ring, it follows from \cite[Proposition 7.2 and Remark 4]{HW}  that $M\in\mathcal{A}_{K_{d-1}}(R)$ if and only if $M^{*}\in\mathcal{B}_{K_{d-1}}(R^{op}),$ and $M\in\mathcal{B}_{K_{d-1}}(R)$ if and only if $M^{*}\in\mathcal{A}_{K_{d-1}}(R^{op})$. In the following proposition, it is checked for a non-commutative ring.

\begin{proposition}\label{3-5-A-BB}
The following assertions hold true:
\begin{enumerate}
\item [\rm (1)] $M\in\mathcal{A}_{K_{d-1}}(R)$ if and only if $M^{*}\in\mathcal{B}_{K_{d-1}}(R^{op})$;
\item [\rm (2)] $M\in\mathcal{B}_{K_{d-1}}(R)$ if and only if $M^{*}\in\mathcal{A}_{K_{d-1}}(R^{op})$.
\end{enumerate}
\end{proposition}
\begin{proof}
(1). $(\Rightarrow)$ Consider the exact sequence  $\mathcal{Y}=\cdots\longrightarrow F_1\longrightarrow F_0\longrightarrow K_{d-1}\longrightarrow 0$ of $R$-modules, where each $F_{i}$ is finitely generated and free. If $M\in\mathcal{A}_{K_{d-1}}(R)$, then $\operatorname{Tor}_{i}^{R}(K_{d-1},M)=0$ for any $i\geq 1.$ Hence by \cite[Lemma 3.53]{Rot2}, $(\mathcal{Y}\otimes_{R}M)^{*}$ is an exact sequence.  So  by \cite[Theorem 2.76 ]{Rot2}, it is easy
to check that $0=\operatorname{Tor}_{i}^{R}(K_{d-1}, M)^{*}\cong \operatorname{Ext}_{R^{op}}^{i}(K_{d-1}, M^{*})$ for any $i\geq 1.$

On the other hand, we have $\operatorname{Ext}_{R}^{i}(K_{d-1}, K_{d-1}\otimes_{R}M)=0 $ for any $i\geq 1$, and  so ${\rm Hom}_{R}(\mathcal{Y},K_{d-1}\otimes_{R}M)$ is exact. By \cite[Lemma 3.53]{Rot2},  we deduce that  ${\rm Hom}_{R}(\mathcal{Y},K_{d-1}\otimes_{R}M)^{*}$ is exact.   By \cite[Lemma 3.55 and Propositions 2.56]{Rot2},  ${\rm Hom}_{R}(\mathcal{Y},K_{d-1}\otimes_{R}M)^{*}\cong (K_{d-1}\otimes_{R}M)^{*}\otimes_{R} \mathcal{Y}\cong\mathcal{Y}\otimes_{R^{op}}(K_{d-1}\otimes_{R}M)^{*} $. So  $\mathcal{Y}\otimes_{R^{op}}{\rm Hom}_{R^{op}}(K_{d-1},M^{*})$ is exact, and then   $\operatorname{Tor}_{i}^{R^{op}}(K_{d-1}, \operatorname{Hom}_{R^{op}}(K_{d-1}, M^{*}))= 0$ for all $i\geq 1$.    On the other hand, we have $M\cong{\rm Hom}_{R}(K_{d-1},K_{d-1}\otimes_{R}M)$. So by \cite[Lemma 3.55 and Propositions 2.56 and 2.76]{Rot2}, $M^{*}\cong{\rm Hom}_{R }(K_{d-1},K_{d-1}\otimes_{R}M)^{*}\cong (K_{d-1}\otimes_{R}M)^{*}\otimes_{R}K_{d-1}\cong K_{d-1}\otimes_{R^{op}}(K_{d-1}\otimes_{R}M)^{*}\cong K_{d-1}\otimes_{R^{op}}{\rm Hom}_{R^{op}}(K_{d-1},M^{*})$.  Then, it follows that $M^{*}\in\mathcal{B}_{K_{d-1}}(R^{op})$.

$(\Leftarrow)$ Consider the exact sequence  $\mathcal{Y}=\cdots\longrightarrow F_1\longrightarrow F_0\longrightarrow K_{d-1}\longrightarrow 0$ of $R^{op}$-modules, where each $F_{i}$ is finitely generated and free. If $M^{*}\in\mathcal{B}_{K_{d-1}}(R^{op})$, then $\operatorname{Ext}^{i}_{R^{op}}(K_{d-1},M^{*})=0$ for any $i\geq 1,$  and so ${\rm Hom}_{R^{op}}(\mathcal{Y},M^{*})$ is exact.
 So by \cite[Theorem 2.76]{Rot2}, $(\mathcal{Y}\otimes_{R}M)^{*}$ is exact and then  by \cite[Lemma 3.53]{Rot2},  $(\mathcal{Y}\otimes_{R}M)$ is exact. So  $\operatorname{Tor}_{i}^{R}(K_{d-1}, M)=0$ for any $i\geq 1$. Also, we have $\operatorname{Tor}_{i}^{R^{op}}(\mathcal{Y}, {\rm Hom}_{R^{op}}(K_{d-1},M^{*}))=0$ for any $i\geq 1$, and then it follows that   $\mathcal{Y}\otimes_{R^{op}}{\rm Hom}_{R^{op}}(K_{d-1},M^{*})$ is exact. Hence by \cite[Theorem 2.76]{Rot2},   $\mathcal{Y}\otimes_{R^{op}}(K_{d-1}\otimes_{R}M)^{*}$ is exact.  Consequently by \cite[Lemma 3.55 and Proposition  2.56]{Rot2},  ${\rm Hom}_{R}(\mathcal{Y},K_{d-1}\otimes_{R}M)^{*}$ is exact, and then  ${\rm Hom}_{R}(\mathcal{Y},K_{d-1}\otimes_{R}M)$ is exact. So $\operatorname{Ext}_{R}^{i}(K_{d-1}, K_{d-1}\otimes_{R}M)=0$ for any $i\geq 1$.  Since $M^{*}\in\mathcal{B}_{K_{d-1}}(R^{op})$, we have $M^{*}\cong K_{d-1}\otimes_{R^{op}}{\rm Hom}_{R^{op}}(K_{d-1},M^{*})\cong K_{d-1}\otimes_{R^{op}}(K_{d-1}\otimes_{R}M)^{*}\cong (K_{d-1}\otimes_{R}M)^{*}\otimes_{R}K_{d-1}\cong {\rm Hom}_{R}(K_{d-1},K_{d-1}\otimes_{R}M)^{*}$, and so $M\cong{\rm Hom}_{R}(K_{d-1},K_{d-1}\otimes_{R}M)$. Then, we get that  $M\in\mathcal{A}_{K_{d-1}}(R)$.

(2). The proof is similar to that of (i).
\end{proof}

\begin{theorem}\label{1.b1}
 The following statements hold.
\begin{enumerate}
\item [\rm (1)]
 $\mathcal{I}^{(n,d)}(S)\subseteq\mathcal{B}_{K_{d-1}}(S)$;  
 \item [\rm (2)] 
 $\mathcal{F}^{(n,d)}(R)\subseteq\mathcal{A}_{K_{d-1}}(R).$
 \end{enumerate}
\end{theorem}
\begin{proof}
(1) If $M\in\mathcal{I}^{(n,d)}(S)$,  then by Lemma \ref{1.b-1}(1), 
 ${\rm Ext}_{S}^{i}(K_{d-1},M)=0$ for any $i\geq 1$. Now,  we show that ${\rm Tor}^{R}_{i}(K_{d-1}, {\rm Hom}_{S}(K_{d-1},M))=0$ for every  any $i\geq 1$. There exists an exact sequence
 $$ \cdots\longrightarrow F_{d+2}\longrightarrow F_{d+1}\longrightarrow F_{d}\longrightarrow K_{d-1}\longrightarrow 0$$ of projective  $R^{op}$-modules, where each $F_j$ is  finitely generated for any $j\geq d$.  On the other hand, ${\rm Ext}^{i}_{R^{op}}(K_{d-1},K_{d-1})=0$ for any $i\geq 1$, so
 we
have the exact sequence
 $$0\longrightarrow {\rm Hom}_{R^{op}}(K_{d-1},K_{d-1})\longrightarrow{\rm Hom}_{R^{op}}(F_{d},K_{d-1})\longrightarrow{\rm Hom}_{R^{op}}(F_{d+1},K_{d-1})\longrightarrow\cdots$$
 of $S$-modules from applying the functor ${\rm Hom}_{R^{op}}(-,K_{d-1})$ to the above exact sequence. Note that
 $S\cong{\rm Hom}_{R^{op}}(K_{d-1},K_{d-1})$, and for all $t\geq d$ there exists an integer $m_{t}$ such that
  ${\rm Hom}_{R^{op}}(F_{t},K_{d-1}) \cong\bigoplus_{l=1}^{m_t}K_{d-1}.$  Therefore by \cite[Proposition 1.7]{ARB}, there is exact sequences 
 $$0\longrightarrow S\longrightarrow\bigoplus_{l=1}^{m_d}K_{d-1}\longrightarrow D\longrightarrow 0$$
 $$0\longrightarrow D_k\longrightarrow\bigoplus_{l=1}^{m_{d+k+2}}K_{d-1}\longrightarrow D_{k+1}\longrightarrow 0$$
 
 of finitely $n$-presented $S$-modules, where for  $k\geq 0$
 $$D={\rm Coker}({\rm Hom}_{R^{op}}(K_{d-1},K_{d-1})\longrightarrow{\rm Hom}_{R^{op}}(F_{d},K_{d-1}))$$ and 
 $$D_{k}={\rm Coker}({\rm Hom}_{R^{op}}(F_{d+k},K_{d-1})\longrightarrow{\rm Hom}_{R^{op}}(F_{d+k+1},K_{d-1})).$$
 Consider the exact sequence
  $$0 \longrightarrow D\longrightarrow \bigoplus_{l=1}^{m_{d+1}}K_{d-1}\longrightarrow \bigoplus_{l=1}^{m_{d+2}}K_{d-1}\longrightarrow \cdots\longrightarrow  \bigoplus_{l=1}^{m_{2d}}K_{d-1}\longrightarrow D_{d-1}\longrightarrow 0.$$
Since $M\in\mathcal{I}^{(n,d)}(S)$, we have
 ${\rm Ext}_{S}^{d+1}( D_{d-1},M)=0$ as $ D_{d-1}$ is finitely $n$-presented $S$-module, and ${\rm Ext}_{S}^{i}(\bigoplus_{l=1}^{m_{t}}K_{d-1},M)=\cong\prod_{l=1}^{m_t}{\rm Ext}_{S}^{i}(K_{d-1},M)=0$ for any $i\geq 1$ by Lemma \ref{1.b-1}(1). It is easy to check that ${\rm Ext}_{S}^{1}( D,M)\cong {\rm Ext}_{S}^{d+1}( D_{d-1},M)=0$.
Step by step, we get that ${\rm Ext}_{S}^{1}(D_k,M)=0$ for any $k\geq 0$, and 
  so we obtain the exact sequence
 sequence
\[\cdots\longrightarrow \operatorname{Hom}_S(\operatorname{Hom}_{R^{op}}(F_d, K_{d-1}), M)\longrightarrow \operatorname{Hom}_S(\operatorname{Hom}_{R^{op}}(K_{d-1}, K_{d-1}), M)\longrightarrow 0.\]
 Thus by \cite[Lemm 3.5]{Rot2}, we have the commutative diagram
\[
\xymatrix{
\cdots\ar[r]&F_d\otimes_{R} \operatorname{Hom}_S(K_{d-1}, M)\ar[d]^{\cong}\ar[r]&K_{d-1}\otimes_{R} \operatorname{Hom}_S(K_{d-1}, M)\ar[d]^{\nu_{M}}\ar[r]&0\\
\cdots\ar[r]&\operatorname{Hom}_S(\operatorname{Hom}_{R^{op}}(F_d, K_{d-1}), M)\ar[r]&M\ar[r]&0
}
\]
from \cite[1.11]{HW} and the fact that $\operatorname{Hom}_S(\operatorname{Hom}_{R^{op}}(K_{d-1}, K_{d-1}), M)\cong \operatorname{Hom}_S(S, M)\cong M$. Hence the sequence
\[\cdots\longrightarrow F_d\otimes_{R} \operatorname{Hom}_S(K_{d-1}, M)\longrightarrow K_{d-1}\otimes_{R} \operatorname{Hom}_S(K_{d-1}, M)\longrightarrow 0\]
is exact and so $\operatorname{Tor}^{R}_{i}(K_{d-1}, \operatorname{Hom}_{S}(K_{d-1}, M))= 0$ for all $i\geq 1$. Also, by the five lemma, the natural evaluation homomorphism $\nu_{M}: K_{d-1}\otimes_{R} \operatorname{Hom}_{S}(K_{d-1}, M)\longrightarrow M$ is an isomorphism. Thus $M\in \mathcal{B}_{K_{d-1}}(S)$ and so $\mathcal{I}^{(n,d)}(S)\subseteq \mathcal{B}_{K_{d-1}}(S)$.

(2)  Let $M\in\mathcal{F}^{(n,d)}(R)$. Then by \cite[Proposition 2.3]{TX}, $M^{*}\in\mathcal{I}^{(n,d)}(R^{op})$. So by (1), $M^{*}\in\mathcal{B}_{K_{d-1}}(R^{op})$. Hence by Proposition \ref{3-5-A-BB}(1), it follows that $M\in\mathcal{A}_{K_{d-1}}(R)$.   
\end{proof}

Let $k$ be a non-negative integer. For convenience, we set
\begin{enumerate}
\item [\rm (i)]
$\mathcal{I}^{(n,d)}(S)_{\leq k}$ = the class of   $S$-modules with 
$(n,d)$-injective dimension at most $k$.
\item [\rm (ii)]
$\mathcal{F}^{(n,d)}(R)_{\leq k}$ = the class of   $R$-modules with 
$(n,d)$-flat dimension at most $k$.
\end{enumerate}
\begin{corollary}\label{2.e1} 
Let
$K_{d-1}$ be a  special faithfully semidualizing bimodule. Then
 $\mathcal{I}^{(n,d)}(S)_{< \infty}\subseteq\mathcal{B}_{k_{d-1}}(S)$ and $\mathcal{F}^{(n,d)}(R)_{< \infty}\subseteq\mathcal{A}_{k_{d-1}}(R).$ 
\end{corollary}
\begin{proof}
It is clear by Theorem \ref{1.b1} and \cite[Theorem 6.3]{HW}.
\end{proof}

The following result plays a fundamental role in this paper. we investigate the relationship between classes $\mathcal{I}_{K_{d-1}}^{(n,d)}(R)$  and  $\mathcal{F}_{K_{m-1}}^{(n,d)}(S)$ with the Auslander class $\mathcal{A}_{K_{d-1}}(R)$ and the Bass class $\mathcal{B}_{K_{d-1}}(S)$, respectively.

\begin{proposition}\label{1.e1}
The following statements hold true:
\begin{enumerate}
\item [\rm (1)]
$M\in\mathcal{I}_{K_{d-1}}^{(n,d)}(R)$ if and only if $M\in\mathcal{A}_{K_{d-1}}(R)$ and $K_{d-1}\otimes_{R}M\in\mathcal{I}^{(n,d)}(S)$;
\item [\rm (2)]
 $N\in\mathcal{F}_{K_{d-1}}^{(n,d)}(S)$ if and only if $N\in\mathcal{B}_{K_{d-1}}(S)$ and ${\rm Hom}_{S}(K_{d-1},N)\in\mathcal{F}^{(n,d)}(R)$.
\end{enumerate}
\end{proposition} 
\begin{proof}
(1) $(\Longrightarrow)$ Let $M\in\mathcal{I}_{K_{d-1}}^{(n,d)}(R)$. Then $M={\rm Hom}_{S}(K_{d-1},I)$ for some $I\in\mathcal{I}^{(n,d)}(S)$. By Theorem \ref{1.b1}(1), $I\in\mathcal{B}_{k_{d-1}}(S)$, and so by \cite[Proposition 4.1]{HW}, $M\in\mathcal{A}_{k_{d-1}}(R)$.
Also $K_{d-1}\otimes_{R}{\rm Hom}_{S}(K_{d-1},I)\cong I,$  and then we get that $$K_{d-1}\otimes_{R}M=K_{d-1}\otimes_{R}{\rm Hom}_{S}(K_{d-1},I)\in\mathcal{I}^{(n,d)}(S).$$

$(\Longleftarrow)$ Let $M\in\mathcal{A}_{K_{d-1}}(R)$ and $ K_{d-1}\otimes_{R}M \in\mathcal{I}^{(n,d)}(S)$. Since ${\rm Hom}_S(K_{d-1},K_{d-1}\otimes_{R}M)\cong M$, we deduce that   $M\in\mathcal{I}_{K_{d-1}}^{(n,d)}(R)$ by Definition \ref{1.a1}.

(2) The proof is similar to that of (1).
\end{proof}
 
\begin{proposition}\label{1.g1}
The following statements hold true:
\begin{enumerate}
\item [\rm (1)]
 $M\in\mathcal{I}_{K_{d-1}}^{(n,d)}(R)$ if and only if $M^*\in\mathcal{F}_{K_{d-1}}^{(n,d)}(R^{op})$;
 \item [\rm (2)]
 $N\in\mathcal{F}_{K_{d-1}}^{(n,d)}(S)$ if and only if $N^*\in\mathcal{I}_{K_{d-1}}^{(n,d)}(S^{op})$.
 \end{enumerate}
 \end{proposition}
\begin{proof}
(1) $(\Longrightarrow)$ Let $M\in\mathcal{I}_{K_{d-1}}^{(n,d)}(R)$. Then $M={\rm Hom}_{S}(K_{d-1},I)$ for some $I\in\mathcal{I}^{(n,d)}(S)$. By \cite[Proposition 3.1]{TX}, $I^*\in \mathcal{F}^{(n,d)}(S^{op})$. Since $K_{d-1}$ is finitely presented,  \cite[Lemma 3.55]{Rot2} implies that  $M^{*}={\rm Hom}_{S}(K_{d-1},I)^{*}\cong I^*\otimes_{S}K_{d-1}$, and so $M^*\in\mathcal{F}_{K_{d-1}}^{(n,d)}(R^{op})$.

$(\Longleftarrow)$ If $M^*\in\mathcal{F}_{K_{d-1}}^{(n,d)}(R^{op})$, then  Proposition \ref{1.e1}(2) implies that $M^{*}\in\mathcal{B}_{K_{d-1}}(R^{op})$ and ${\rm Hom}_{R^{op}}(K_{d-1},M^{*})\in\mathcal{F}^{(n,d)}(S^{op})$. By  Proposition \ref{3-5-A-BB}(1), it follows that $M\in\mathcal{A}_{K_{d-1}}(R)$.
Also, by \cite[Theorem 2.76]{Rot2}, ${\rm Hom}_{R^{op}}(K_{d-1},M^{*})\cong (K_{d-1}\otimes_{R}M)^*$. So by \cite[Proposition 3.1]{TX}, we get that $K_{d-1}\otimes_{R}M\in\mathcal{I}^{(n,d)}(S)$, and consequently by Proposition \ref{1.e1}(1), $M\in\mathcal{I}_{K_{d-1}}^{(n,d)}(R)$. 

(2) It is similar to the proof of (1) using \cite[Proposition 2.3]{TX} and Proposition \ref{3-5-A-BB}(2).
\end{proof}
\begin{corollary}\label{1.mu1}
The following statements hold true:
\begin{enumerate}
\item [\rm (1)]
 $M\in\mathcal{I}_{K_{d-1}}^{(n,d)}(R)$ if and only if $M^{**}\in\mathcal{I}_{K_{d-1}}^{(n,d)}(R)$;
 \item [\rm (2)]
 $N\in\mathcal{F}_{K_{d-1}}^{(n,d)}(S)$ if and only if $N^{**}\in\mathcal{F}_{K_{d-1}}^{(n,d)}(S)$.
 \end{enumerate}
 \end{corollary}
\begin{proof}
It is clear by Proposition \ref{1.g1}.
\end{proof}

\begin{corollary}\label{1.k1}
The following statements hold.
\begin{enumerate}
\item [\rm (1)]
 $ K_{d-1}\otimes_{R}N \in\mathcal{F}_{K_{d-1}}^{(n,d)}(S)$ if and only if $N\in\mathcal{F}^{(n,d)}(R)$;
 \item [\rm (2)]
 ${\rm Hom}_{S}(K_{d-1},M)\in\mathcal{I}_{K_{d-1}}^{(n,d)}(R)$ if and only if $M\in\mathcal{I}^{(n,d)}(S)$.
 \end{enumerate}
 \end{corollary}
\begin{proof}
(1) $(\Longrightarrow)$ If $ K_{d-1}\otimes_{R}N \in\mathcal{F}_{K_{d-1}}^{(n,d)}(S)$, then by Proposition \ref{1.e1}(2), $ K_{d-1}\otimes_{R}N \in\mathcal{B}_{K_{d-1}}(S)$. Hence by replacing $K_{d-1}$ instead $C$ from \cite[Lemma 2.9]{Z.TT}, $N\in\mathcal{A}_{K_{d-1}}(R)$. Also by Proposition \ref{1.e1}(2), we observe that
${\rm Hom}_{S}(K_{d-1},K_{d-1}\otimes_{R}N)\in\mathcal{F}^{(n,d)}(R)$. On the other hand,   $N\cong{\rm Hom}_{S}(K_{d-1},K_{d-1}\otimes_{R}N)$, since $N\in\mathcal{A}_{K_{d-1}}(R)$. Consequently $N\in\mathcal{F}^{(n,d)}(R)$.

$(\Longleftarrow)$ is obvious.

(2) If ${\rm Hom}_{S}(K_{d-1},M)\in\mathcal{I}_{K_{d-1}}^{(n,d)}(R)$, then by Proposition \ref{1.g1}(1), ${\rm Hom}_{S}(K_{d-1},M)^{*}\in\mathcal{F}_{K_{d-1}}^{(n,d)}(R^{op})$. By \cite[Lemma 3.55 and Proposition 2.56]{Rot2}, ${\rm Hom}_{S}(K_{d-1},M)^{*}\cong M^{*}\otimes_{S}K_{d-1}\cong K_{d-1}\otimes_{S^{op}}M^{*}$. So by (1), $M^{*}\in\mathcal{F}^{(n,d)}(S^{op})$, and then by \cite[Proposition 3.1]{TX}, 
$M\in\mathcal{I}^{(n,d)}(S)$.
\end{proof}
In the next proposition, we show that classes $\mathcal{I}_{K_{d-1}}^{(n,d)}(R)$ (resp.  $\mathcal{F}_{K_{d-1}}^{(n,d)}(S)$) is closed under extensions.
 \begin{proposition}\label{1.f1}
 The following assertions hold:
\begin{enumerate}
\item [\rm (1)]
If  $0\rightarrow M\rightarrow N\rightarrow L\rightarrow 0$ is a short exact sequence of $R$-modules and 
 $M\in\mathcal{I}_{K_{d-1}}^{(n,d)}(R)$, then $N\in\mathcal{I}_{K_{d-1}}^{(n,d)}(R)$ if   $L\in\mathcal{I}_{K_{d-1}}^{(n,d)}(R)$;
 \item [\rm (2)]
 If  $0\rightarrow M\rightarrow N\rightarrow L\rightarrow 0$ is a short exact sequence of $S$-modules and 
 $M\in\mathcal{F}_{K_{d-1}}^{(n,d)}(S)$, then $N\in\mathcal{F}_{K_{d-1}}^{(n,d)}(S)$ if   $L\in\mathcal{F}_{K_{d-1}}^{(n,d)}(S)$.
 \end{enumerate}
 \end{proposition}
\begin{proof}
(1)  
 Let $M, L\in\mathcal{I}_{K_{d-1}}^{(n,d)}(R)$. Then by Proposition \ref{1.e1}(1), $M,L\in\mathcal{A}_{K_{d-1}}(R)$ and also $K_{d-1}\otimes_{R}M$ and $K_{d-1}\otimes_{R}L$ are in $\mathcal{I}^{(n,d)}(S)$. Hence  by \cite[Theorem 6.2]{HW}, it follows that $N\in\mathcal{A}_{K_{d-1}}(R)$. On the other hand, since $L\in\mathcal{A}_{K_{d-1}}(R)$, ${\rm Tor}^{R}_{i}(K_{d-1},L)=0$ for any $i\geq 1$. So, there exists the following exact sequence of $S$-modules:
$$0\longrightarrow K_{d-1}\otimes_{R}M\longrightarrow K_{d-1}\otimes_{R}N\longrightarrow K_{d-1}\otimes_{R}L\longrightarrow 0.$$
If $U$ is a  finitely $n$-presented $S$-module, then we have the following exact sequence:
$$0= {\rm Ext}_{S}^{d+1}(U,K_{d-1}\otimes_{R}M)\longrightarrow {\rm Ext}_{S}^{d+1}(U,K_{d-1}\otimes_{R}N)\longrightarrow {\rm Ext}_{S}^{d+1}(U,K_{d-1}\otimes_{R}L)=0.$$
Consequently ${\rm Ext}_{S}^{d+1}(U,K_{d-1}\otimes_{R}N)=0$, and so $ K_{d-1}\otimes_{R}N \in\mathcal{I}^{(n,d)}(S)$ and hence by Proposition \ref{1.e1}(1), we get that $N\in\mathcal{I}_{K_{d-1}}^{(n,d)}(R)$.

(2) By Proposition \ref{1.g1}, it is clear.
\end{proof}

The class $\mathcal{I}_{K_{d-1}}^{(n,d)}(R)$ (resp.  $\mathcal{F}_{K_{d-1}}^{(n,d)}(S)$) is
closed under direct summands, direct
products and direct sums, see the  propositions \ref{1.h1}, \ref{1.h15} and \ref{1.h16}.
\begin{proposition}\label{1.h1}
 The classes $\mathcal{I}_{K_{d-1}}^{(n,d)}(R)$ and $\mathcal{F}_{K_{d-1}}^{(n,d)}(S)$ are
closed under direct summands.
 \end{proposition}
\begin{proof}
Suppose that $M\in\mathcal{I}_{K_{d-1}}^{(n,d)}(R)$ and $R$-module $N$ is a summand of $M$. Then, there is a submodule $L$ of $M$ such that $M=L\oplus N$.
Hence there is a split exact sequence $0\rightarrow L\rightarrow M\rightarrow N\rightarrow 0$. So, the split exact sequence $$0\rightarrow K_{d-1}\otimes_{R}L\rightarrow K_{d-1}\otimes_{R}M\rightarrow K_{d-1}\otimes_{R}N\rightarrow 0$$ of $S$-modules exists. Hence we have $K_{d-1}\otimes_{R}M=(K_{d-1}\otimes_{R}L)\oplus (K_{d-1}\otimes_{R}N)$. By Proposition \ref{1.e1}(1), $ K_{d-1}\otimes_{R}M \in\mathcal{I}^{(n,d)}(S)$, since $M\in\mathcal{I}_{K_{d-1}}^{(n,d)}(R)$. Hence by \cite[Proposition 2.10]{ZX}, $ K_{d-1}\otimes_{R}L $ and $ K_{d-1}\otimes_{R}N $ are in $\mathcal{I}^{(n,d)}(S)$. Consequently by Proposition \ref{1.e1}(1), we deduce that 
$L,N\in\mathcal{I}_{K_{d-1}}^{(n,d)}(R)$. Similarly, it follows that  the class $\mathcal{F}_{K_{d-1}}^{(n,d)}(S)$ is
closed under direct summands.
\end{proof}
\begin{proposition}\label{1.h15}
The following statements are equivalent:
\begin{enumerate}
\item [\rm (1)]
$M_j\in\mathcal{I}_{K_{d-1}}^{(n,d)}(R)$ for any $j\in J$;
 \item [\rm (2)]
$\prod_{j\in J}M_i\in\mathcal{I}_{K_{d-1}}^{(n,d)}(R)$;
 \item [\rm (3)]
$\bigoplus_{j\in J}M_J\in\mathcal{I}_{K_{d-1}}^{(n,d)}(R)$.
 \end{enumerate}
 \end{proposition}
\begin{proof}
(1)$\Longrightarrow$(2)
By Proposition \ref{1.e1}(1), $M_j\in\mathcal{A}_{K_{d-1}}(R)$ and $ K_{d-1}\otimes_{R}M_j \in\mathcal{I}^{(n,d)}(S)$ for any $j\in J$. By \cite[Proposition 2.2(2)]{TX}, $\prod_{j\in J}(K_{d-1}\otimes_{R}M_j)\in\mathcal{I}^{(n,d)}(S)$.  \cite[Lemma 2.10]{NDING} implies that  $\prod_{j\in J}(K_{d-1}\otimes_{R}M_j)\cong K_{d-1}\otimes_{R}(\prod_{j\in J}M_j)$, since $K_{d-1}$ is finitely presented. So, $ K_{d-1}\otimes_{R}(\prod_{j\in J}M_j)\in\mathcal{I}^{(n,d)}(S)$. Also by replacing $K_{d-1}$ instead $_SC_R$ from \cite[Proposition 4.2]{HW}, $\prod_{j\in J}M_j\in\mathcal{A}_{K_{d-1}}(R)$, and hence by Proposition \ref{1.e1}(1), we obtain that $\prod M_{j\in J}\in\mathcal{I}_{K_{d-1}}^{(n,d)}(R).$

(2)$\Longrightarrow$(1) Let $\prod_{j\in J}M_j\in\mathcal{I}_{K_{d-1}}^{(n,d)}(R)$. Then by Proposition \ref{1.e1}(1), $\prod_{i\in J}M_j\in\mathcal{A}_{K_{d-1}}(R)$ and $\prod_{i\in J}(K_{d-1}\otimes_{R}M_j)\in\mathcal{I}^{(n,d)}(S)$. So by \cite[Proposition 2.10]{ZX}, $ K_{d-1}\otimes_{R}M_j \in\mathcal{I}^{(n,d)}(S)$ and then for any $j\in J$, $ K_{d-1}\otimes_{R}M_j \in\mathcal{B}_{K_{d-1}}(S)$ by Theorem \ref{1.b1}(1).  Hence by replacing $K_{d-1}$ instead $_SC_R$ from \cite[Lemma 3.9(2)]{WGZ}, we deduce that $M_j\in\mathcal{A}_{K_{d-1}}(R)$ and so by Proposition \ref{1.e1}(1), $M_j\in\mathcal{I}_{K_{d-1}}^{(n,d)}(R)$
for any $j\in J$.

(1)$\Longrightarrow$(3) By Proposition \ref{1.e1}(1), $M_j\in\mathcal{A}_{K_{d-1}}(R)$ and $ K_{d-1}\otimes_{R}M_j \in\mathcal{I}^{(n,d)}(S)$ for any $j\in J$. So by \cite[Proposition 2.10]{ZX}, $\bigoplus_{j\in J}(K_{d-1}\otimes_{R}M_j)\in\mathcal{I}^{(n,d)}(S)$. Also by \cite[Theorem 2.65]{Rot2}, $\bigoplus_{j\in J}(K_{d-1}\otimes_{R}M_j)\cong K_{d-1}\otimes_{R}(\bigoplus_{j\in J}M_j)$, and then $ K_{d-1}\otimes_{R}(\bigoplus_{j\in J}M_j)\in\mathcal{I}^{(n,d)}(S)$. By replacing $K_{d-1}$ instead $_SC_R$ from \cite[Proposition 4.2]{HW}, $\bigoplus_{j\in J}M_j\in\mathcal{A}_{K_{d-1}}(R)$, and so by Proposition \ref{1.e1}(1), we get that $\bigoplus M_{j\in J}\in\mathcal{I}_{K_{d-1}}^{(n,d)}(R).$ 

(3)$\Longrightarrow$(1) The proof is similar to that of (2)$\Longrightarrow$(1).
\end{proof}
\begin{proposition}\label{1.h16}
The following statements are equivalent:
\begin{enumerate}
\item [\rm (1)]
$M_j\in\mathcal{F}_{K_{d-1}}^{(n,d)}(S)$ for any $j\in J$;
 \item [\rm (2)]
$\prod_{j\in J}M_i\in\mathcal{F}_{K_{d-1}}^{(n,d)}(S)$;
 \item [\rm (3)]
$\bigoplus_{j\in J}M_J\in\mathcal{F}_{K_{d-1}}^{(n,d)}(S)$.
 \end{enumerate}
 \end{proposition}
\begin{proof}
The proof is similar to that of Proposition \ref{1.h15}.
\end{proof}
\begin{corollary}\label{1.l-1}
The following statements hold.
\begin{enumerate}
\item [\rm (1)]
  $M\in\mathcal{I}_{K_{d-1}}^{(n,d)}(R)$ if and only if every pure submodule and pure epimorphic image of $M$ is in $\mathcal{I}_{K_{d-1}}^{(n,d)}(R)$;
\item [\rm (2)]
 $M\in\mathcal{F}_{K_{d-1}}^{(n,d)}(S)$ if and only if every pure submodule and pure epimorphic image of $M$ is in $\mathcal{F}_{K_{d-1}}^{(n,d)}(S)$. 
 \end{enumerate}
 \end{corollary}
\begin{proof}
(1) Suppose that  $M\in\mathcal{I}_{K_{d-1}}^{(n,d)}(R)$ and $N$ is a pure submodule of $M$. Then there exists a pure exact sequence $0 \rightarrow N \rightarrow M \rightarrow M / N \rightarrow 0$ which gives rise to a split exact sequence $0 \rightarrow(M / N)^* \rightarrow M^* \rightarrow N^* \rightarrow 0$ of $R^{op}$-modules.  By Proposition \ref{1.g1}(1), $M^*$ is in $\mathcal{F}_{K_{d-1}}^{(n,d)}(R^{op})$. Then by Proposition \ref{1.h16}, $M^*$ is in $\mathcal{F}_{K_{d-1}}^{(n,d)}(R^{op})$ if and only if $N^*$ and $(M / N)^*$ are in $\mathcal{F}_{K_{d-1}}^{(n,d)}(R^{op})$. Hence by Propositions \ref{1.g1}(1) and \ref{1.h15}, we deduce that $M$ is in $\mathcal{I}_{K_{d-1}}^{(n,d)}(R)$ if and only if $N$ and $M / N$ are in $\mathcal{I}_{K_{d-1}}^{(n,d)}(R)$. 

(2) It is similar to the proof of (1).
\end{proof}
\ \ Let $\mathcal{X}$ be a class of $R$-modules and $M$ be an 
 $R$-module. Following \cite{EM}, we say that a  morphism $f : F\rightarrow M$ is a
$\mathcal{X}$-precover of $M$ if $F\in\mathcal{X}$ and ${\rm Hom}_{R}(F^{'}, F) \rightarrow {\rm Hom}_{R}(F^{'},M)\rightarrow 0$ is exact
for all $F^{'}\in\mathcal{X}$. Moreover, if whenever a  morphism $g : F\rightarrow F$ such that
$fg = f$ is an automorphism of $F$, then $f : F\rightarrow M$ is called a  $\mathcal{X}$-cover of $M$. The
class $\mathcal{X}$ is called (pre)covering if each $R$-module has a $\mathcal{X}$-(pre)cover. Dually,
the notions of $\mathcal{X}$-preenvelopes, $\mathcal{X}$-envelopes and (pre)enveloping classes are defined.

A duality pair over $R$ \cite{HJ} is a pair $(\mathcal{M}, \mathcal{N})$, where $\mathcal{M}$ is a class of
$R$-modules and $\mathcal{N}$ is a class of $R^{op}$-
modules, subject to the following conditions:
(1) For an $R$-module $M$, one has $M\in \mathcal{M}$ if and only if $M^{*}\in \mathcal{N}$.
(2) $\mathcal{N}$ is closed under direct summands and finite direct sums.

In the following theorem ,  we show that   the classes $\mathcal{I}_{K_{d-1}}^{(n,d)}(R)$ and $\mathcal{F}_{K_{d-1}}^{(n,d)}(S)$ are preenveloping and covering.
\begin{theorem}\label{1.z1}
The following statements hold.
\begin{enumerate}
\item [\rm (1)]
  The pair $(\mathcal{I}_{ K_{d-1}}^{(n,d)}(R), \mathcal{F}_{ K_{d-1}}^{(n,d)}(R^{op}))$ is a duality pair, and the class $\mathcal{I}_{ K_{d-1}}^{(n,d)}(R)$  is  covering and preenveloping;
\item [\rm (2)]
 The pair $(\mathcal{F}_{ K_{d-1}}^{(n,d)}(S), \mathcal{I}_{ K_{d-1}}^{(n,d)}(S^{op}))$ is a duality pair, and the class $\mathcal{F}_{ K_{d-1}}^{(n,d)}(S)$  is  covering and preenveloping.
 \end{enumerate}
\end{theorem}
\begin{proof}
(1) By Propositions \ref{1.h1} and \ref{1.h15}, class $\mathcal{F}_{ K_{d-1}}^{(n,d)}(R^{op})$ is closed under direct summands and 
 direct sums.
By Proposition \ref{1.g1}(1), $M\in\mathcal{I}_{ K_{d-1}}^{(n,d)}(R)$ if and only if $M^{*}\in\mathcal{F}_{ K_{d-1}}^{(n,d)}(R^{op})$, and so we conclude that $(\mathcal{I}_{ K_{d-1}}^{(n,d)}(R), \mathcal{F}_{ K_{d-1}}^{(n,d)}(R^{op}))$ is a duality pair. Therefore, from Corollary \ref{1.l-1} and \cite[Theorem 3.1]{HJ}, the class $\mathcal{I}_{ K_{d-1}}^{(n,d)}(R)$  is  covering and preenveloping.

(2) The proof is similar to the proof of (1) by using Propositions \ref{1.g1}(2), \ref{1.h1}, \ref{1.h16}, Corollary \ref{1.l-1} and \cite[Theorem 3.1]{HJ}.
\end{proof}
\section{Foxby equivalence under special semidualizing bimodules}
\ \ \ \ In this section, we investigate Foxby equivalence relative to the classes $\mathcal{I}_{ K_{d-1}}^{(n,d)}(R)$  and $\mathcal{F}_{ K_{d-1}}^{(n,d)}(S)$. Then over $n$-coherent rings, we give homological behavior of the classes $\mathcal{I}_{ K_{d-1}}^{(n,d)}(R)_{<\infty}$  and $\mathcal{F}_{ K_{d-1}}^{(n,d)}(S)_{<\infty}$ with respect to extentions, kernels of epimorphisms and cokernels of monomorphisms.

\begin{proposition}\label{2.a1}
Tere are equivalences of
categories:

$
\xymatrix@C=3cm{
 (1) \ \mathcal{I}_{ K_{d-1}}^{(n,d)}(R)\ar@ <0.9ex>[r]^{\ \ \ \ \ \  K_{d-1}\otimes_{R^{-}}}_{\ \ \ \ \sim}& \ar@ <0.9ex>[l]^{\ \ \ \ {\rm Hom}_{S}(K_{d-1},-)} \mathcal{I}^{(n,d)}(S); \\
}
$

$
\xymatrix@C=3cm{
(2) \ \mathcal{F}^{(n,d)}(R)\ar@ <0.9ex>[r]^{\ \ \ \ K_{d-1}\otimes_{R^{-}}}_{\sim}& \ar@ <0.9ex>[l]^{\ \ \ \ {\rm Hom}_{S}(K_{d-1},-)} \mathcal{F}_{ K_{d-1}}^{(n,d)}(S). \\
}
$
\end{proposition}
\begin{proof}
 We consider that  the functor ${\rm Hom}_{S}(K_{d-1},-)$ maps $\mathcal{I}^{(n,d)}(S)$ to $\mathcal{I}_{ K_{d-1}}^{(n,d)}(R)$ by Definition \ref{1.a1},  and by Proposition \ref{1.e1}(1), the functor $K_{d-1}\otimes_{R}-$ maps   $\mathcal{I}_{ K_{d-1}}^{(n,d)}(R)$ to $\mathcal{I}^{(n,d)}(S)$. So, if $M\in\mathcal{I}^{(n,d)}(S)$, then by Theorem \ref{1.b1},  $M\in\mathcal{B}_{K_{d-1}}(S)$, and if $N\in\mathcal{I}_{ K_{d-1}}^{(n,d)}(R)$, then by Proposition \ref{1.e1}(1), $N\in\mathcal{A}_{ K_{d-1}}(R)$. Hence we have natural isomorphisms $M\cong K_{d-1}\otimes_{R}{\rm Hom}_{S}(K_{d-1},M)$ and $N\cong{\rm Hom}_{S}(K_{d-1},K_{d-1}\otimes_{R}N).$  Dually, we get the second one.
\end{proof}

\begin{definition}\label{2.c1}
{\rm Let $K_{d-1}$ be a special   faithfully semidualizing bimodule. Then,
the $K_{d-1}$-$(n,d)$-injective dimension of an $R$-module $M$ and $K_{d-1}$-$(n,d)$-flat dimension of an $S$-module $N$
are defined by 
$K_{d-1}$-$(n,d)$-${\rm id}_{R}(M)\leq k$ if  there exists an exact sequence  $$0\longrightarrow M\longrightarrow {\rm Hom}_{S}(K_{d-1},I_0) \longrightarrow \cdots\longrightarrow {\rm Hom}_{S}(K_{d-1},I_k)\longrightarrow0$$ of $R$-modules, where each $I_i\in\mathcal{I}^{(n,d)}(S)$, and 
$K_{d-1}$-$(n,d)$-${\rm fd}_{S}(N)\leq k$ if  there exists an exact sequence  $$0\longrightarrow K_{d-1}\otimes_{R}F_k\longrightarrow K_{d-1}\otimes_{R}F_{k-1} \longrightarrow\cdots\longrightarrow K_{d-1}\otimes_{R}F_0\longrightarrow N\longrightarrow 0$$ of $S$-modules, where each $F_i\in\mathcal{F}^{(n,d)}(R)$.}
\end{definition}
If $k=0$, then  $M$ and $N$ are $K_{d-1}$-$(n,d)$-injective and $K_{d-1}$-$(n,d)$-flat, respectively.
 We denote by $\mathcal{I}_{K_{d-1}}^{(n,d)}(R)_{\leq k}$ and $\mathcal{F}_{K_{d-1}}^{(n,d)}(S)_{\leq k}$ the classes of  $R$-modules with 
$K_{d-1}$-$(n,d)$-injective dimension and $S$-modules with $K_{d-1}$-$(n,d)$-flat dimension  at most $k$, respectively. 

The next result is a component of the Foxby equivalence, (see Theorem \ref{2.r1}).
\begin{proposition}\label{2.o1}
There is equivalence of
categories:
\begin{displaymath}
\xymatrix@C=3cm{
 \mathcal{A}_{ K_{d-1}}(R)\ar@ <0.8ex>[r]^{\ \ \ \ \ \  K_{d-1}\otimes_{R^{-}}}_{\ \ \ \ \sim}& \ar@ <0.8ex>[l]^{\ \ \ \ {\rm Hom}_{S}(K_{d-1},-)} \mathcal{B}_{K_{d-1}}(S)}
\end{displaymath}
\end{proposition}
\begin{proof}
By replacing $K_{d-1}$ instead $C$ from \cite[Proposition 4.1]{HW}  follow.
\end{proof}
\begin{proposition}\label{2.d1}
Let $K_{d-1}$ be a  special faithfully semidualizing bimodule. Then, there are equivalences of
categories:

$
\xymatrix@C=3cm{
 (1) \ \ \ \ \ \ \ \mathcal{I}_{ K_{d-1}}^{(n,d)}(R)_{\leq k}\ar@ <0.9ex>[r]^{\ \ \ \ \ \  K_{d-1}\otimes_{R^{-}}}_{\ \ \ \ \sim}& \ar@ <0.9ex>[l]^{\ \ \ \ {\rm Hom}_{S}(K_{d-1},-)} \mathcal{I}^{(n,d)}(S)_{\leq k}; \\
}
$

$
\xymatrix@C=3cm{
(2)  \ \ \ \ \ \ \ \  \mathcal{F}^{(n,d)}(R)_{\leq k}\ar@ <0.9ex>[r]^{\ \ \ \ K_{d-1}\otimes_{R^{-}}}_{\sim}& \ar@ <0.9ex>[l]^{\ \ \ \ {\rm Hom}_{S}(K_{d-1},-)} \mathcal{F}_{ K_{d-1}}^{(n,d)}(S)_{\leq k}. \\
}
$
\end{proposition}
\begin{proof}
(1) By proposition \ref{2.a1}(1), it is clear for $k=0$. Assume that $k\geq1$ and $M\in\mathcal{I}^{(n,d)}(S)_{\leq k}$. Then, there is an exact sequence $$0\longrightarrow M\longrightarrow I_0 \longrightarrow I_1 \longrightarrow\cdots\longrightarrow I_k\longrightarrow0$$ of $S$-modules, where each $I_j\in\mathcal{I}^{(n,d)}(S)$ for any $0\leq j\leq k$. It follows that $D_{j-1}\in\mathcal{I}^{(n,d)}(S)_{\leq k-j}$, where $D_{j}={\rm Coker}(I_{j-1}\rightarrow I_{j})$. Thus by Corollary \ref{2.e1}, $ D_{j-1}\in\mathcal{B}_{k_{d-1}}(S)$  for any $0\leq j\leq k$, and so we have  ${\rm Ext}_{S}^{i}(K_{d-1},D_{j-1})=0= {\rm Ext}_{S}^{i}(K_{d-1},M)$ for any $i\geq 1$ and $0\leq j\leq k$.
Thus, we obtain an exact sequence
$$0\longrightarrow {\rm Hom}_{S}(K_{d-1},M)\longrightarrow {\rm Hom}_{S}(K_{d-1},I_0)  \longrightarrow\cdots\longrightarrow {\rm Hom}_{S}(K_{d-1},I_k)\longrightarrow0$$ of  $R$-modules, where ${\rm Hom}_{S}(K_{d-1},I_j)\in\mathcal{I}_{K_{d-1}}^{(n,d)}(R)$, and so we deduce that ${\rm Hom}_{S}(K_{d-1},M)\in\mathcal{I}_{ K_{d-1}}^{(n,d)}(R)_{\leq k}.$

Conversely, let $N\in\mathcal{I}_{ K_{d-1}}^{(n,d)}(R)_{\leq k}$. Then  we have the following exact sequence of $R$-modules:
$$0\longrightarrow N\longrightarrow {\rm Hom}_{S}(K_{d-1},I_0)  \longrightarrow\cdots\longrightarrow {\rm Hom}_{S}(K_{d-1},I_k)\longrightarrow0,$$
where every $I_j\in\mathcal{I}^{(n,d)}(S)$ for any $0\leq j\leq k$. By Theorem \ref{1.b1}(1), $I_j\in\mathcal{B}_{K_{d-1}}(S)$, and hence by Proposition \ref{2.o1}, we get that ${\rm Hom}_{S}(K_{d-1},I_j)\in\mathcal{A}_{K_{d-1}}(R)$. Then by \cite[Theorem 6.2]{HW}, it follows that ${\rm ker}({\rm Hom}_{S}(K_{d-1},I_j)  \rightarrow{\rm Hom}_{S}(K_{d-1},I_{j+1}))\in\mathcal{A}_{K_{d-1}}(R)$. So we obtain the following exact sequence
$$0\rightarrow K_{d-1}\otimes_{R}N\rightarrow K_{d-1}\otimes_{R}{\rm Hom}_{S}(K_{d-1},I_0)  \rightarrow\cdots\rightarrow K_{d-1}\otimes_{R}{\rm Hom}_{S}(K_{d-1},I_k)\rightarrow0.$$
Also since $I_j\in\mathcal{B}_{K_{d-1}}(S)$, we have $K_{d-1}\otimes_{R}{\rm Hom}_{S}(K_{d-1},I_j)\cong I_{j}$, and then we get the following exact sequence
$$0\longrightarrow K_{d-1}\otimes_{R}N\longrightarrow I_0 \longrightarrow\cdots\longrightarrow I_k\longrightarrow0,$$
 and consequently $(K_{d-1}\otimes_{R}N)\in\mathcal{I}^{(n,d)}(S)_{\leq k}$.
  So for every $M\in\mathcal{I}^{(n,d)}(S)_{\leq k}$ and every $N\in\mathcal{I}_{ K_{d-1}}^{(n,d)}(R)_{\leq k}$, we deduce that $M\cong K_{d-1}\otimes_{R}{\rm Hom}_{S}(K_{d-1},M)$ and $N\cong {\rm Hom}_{S}(K_{d-1},K_{d-1}\otimes_{R}N).$

 (2) Let $M\in\mathcal{F}_{K_{d-1}}^{(n,d)}(S)_{\leq k}$. Then by Proposition \ref{1.g1}(2), $M^{*}\in\mathcal{I}_{K_{d-1}}^{(n,d)}(S^{op})_{\leq k}$. So by (1), $K_{d-1}\otimes_{S^{op}}M^{*}\cong M^{*}\otimes_{S}K_{d-1}\in\mathcal{I}^{(n,d)}(R^{op})_{\leq k}$. By \cite[Lemma 3.55]{Rot2}, $M^{*}\otimes_{S}K_{d-1}\cong {\rm Hom}_{S}(K_{d-1},M)^{*}$. Hence by \cite[Proposition 3.1]{TX} and Corollary \ref{1.l-1}(2), we can conclude that ${\rm Hom}_{S}(K_{d-1},M)\in\mathcal{F}^{(n,d)}(R)_{\leq k}$. If $N\in\mathcal{F}^{(n,d)}(R)_{\leq k}$, then by \cite[Proposition 2.3]{TX}, $N^{*}\in\mathcal{I}^{(n,d)}(R^{op})_{\leq k}$. So by (1),  ${\rm Hom}_{R^{op}}(K_{d-1},N^{*})\in\mathcal{I}_{K_{d-1}}^{(n,d)}(S^{op})_{\leq k}$. Since ${\rm Hom}_{R^{op}}(K_{d-1},N^{*})\cong (K_{d-1}\otimes_{R}N)^{*}$, we get that $K_{d-1}\otimes_{R}N\in  \mathcal{F}_{K_{d-1}}^{(n,d)}(S)_{\leq k}$ by Proposition \ref{1.g1}(2). 
\end{proof}
\begin{proposition}\label{2.t1}
Let $K_{d-1}$ be a  special faithfully semidualizing bimodule. Then the following statements hold.
\begin{enumerate}
\item [\rm (1)]
$\mathcal{I}_{ K_{d-1}}^{(n,d)}(R)_{\leq k}\subseteq\mathcal{A}_{ K_{d-1}}(R)$;
\item [\rm (2)]
$\mathcal{F}_{ K_{d-1}}^{(n,d)}(S)_{\leq k}\subseteq\mathcal{B}_{ K_{d-1}}(S)$.
 \end{enumerate}
\end{proposition}
\begin{proof}
(1) Let $M\in \mathcal{I}_{ K_{d-1}}^{(n,d)}(R)_{\leq k}$. If $k=0$, then $M\in\mathcal{I}_{ K_{d-1}}^{(n,d)}(R)$ and so by Proposition \ref{1.e1}(1), $M\in\mathcal{A}_{ K_{d-1}}(R)$.
If $k\geq 1$, then  there exists an exact sequence
$$0\longrightarrow M\longrightarrow {\rm Hom}_{S}(K_{d-1},I_0) \longrightarrow {\rm Hom}_{S}(K_{d-1},I_1) \longrightarrow\cdots\longrightarrow {\rm Hom}_{S}(K_{d-1},I_k)\longrightarrow0$$ of $R$-modules, where each $I_j\in\mathcal{I}^{(n,d)}(S)$ for any $0\leq j\leq k$. Every ${\rm Hom}_{S}(K_{d-1},I_j) \in\mathcal{I}_{ K_{d-1}}^{(n,d)}(R)$, and hence  \cite[Theorem 6.2]{HW} implies that $M\in\mathcal{A}_{ K_{d-1}}(R)$.

(2) Let $N\in \mathcal{F}_{ K_{d-1}}^{(n,d)}(S)_{\leq k}$. Then by Proposition \ref{1.g1}(2), $N^{*}\in \mathcal{I}_{ K_{d-1}}^{(n,d)}(S^{op})_{\leq k}$, and so by (1), $N^{*}\in\mathcal{A}_{ K_{d-1}}(S^{op})$. Hence Proposition \ref{3-5-A-BB}(2) implies that $N\in\mathcal{B}_{ K_{d-1}}(S)$.
\end{proof}
Using Theorem \ref{1.b1}, Popositions \ref{2.a1}, \ref{2.o1}, \ref{2.d1} and \ref{2.t1},  one of the main results is obtained as follows.
\begin{theorem}{\rm (Foxby Equivalence)}\label{2.r1}
Let $K_{d-1}$ be a special faithfully semidualizing bimodule. Then, there is equivalences of
categories:
\begin{displaymath}
\xymatrix@C=3cm{
\mathcal{F}^{(n,d)}(R)\ar@ <0.9ex>[r]^{K_{d-1}\otimes_{R^{-}}}_{\sim}\ar@{^{(}->}[d]& \ar@ <0.9ex>[l]^{Hom_{S}(K_{d-1},-)} \mathcal{F}_{ K_{d-1}}^{(n,d)}(S)\ar@{^{(}->}[d] \\
\mathcal{F}^{(n,d)}(R)_{\leq k}\ar@ <0.9ex>[r]^{K_{d-1}\otimes_{R^{-}}}_{\sim}\ar@{^{(}->}[d]& \ar@ <0.9ex>[l]^{Hom_{S}(K_{d-1},-)} \mathcal{F}_{ K_{d-1}}^{(n,d)}(S)_{\leq k}\ar@{^{(}->}[d] \\
\mathcal{A}_{K_{d-1}}(R)\ar@ <0.9ex>[r]^{K_{d-1}\otimes_{R^{-}}}_{\sim}& \ar@ <0.9ex>[l]^{Hom_{S}(K_{d-1},-)} \mathcal{B}_{K_{d-1}}(S) \\
\mathcal{I}_{ K_{d-1}}^{(n,d)}(R)_{\leq k}\ar@ <0.9ex>[r]^{K_{d-1}\otimes_{R^{-}}}_{\sim}\ar@{^{(}->}[u]& \ar@ <0.9ex>[l]^{Hom_{S}(K_{d-1},-)} \mathcal{I}^{(n,d)}(S)_{\leq k}\ar@{^{(}->}[u] \\
\mathcal{I}_{ K_{d-1}}^{(n,d)}(R)\ar@ <0.9ex>[r]^{K_{d-1}\otimes_{R^{-}}}_{\sim}\ar@{^{(}->}[u]& \ar@ <0.9ex>[l]^{Hom_{S}(K_{d-1},-)} \mathcal{I}^{(n,d)}(S)\ar@{^{(}->}[u] .
}
\end{displaymath}
\end{theorem}
\begin{corollary}\label{1.ee1}
Let $K_{d-1}$ be a special faithfully semidualizing bimodule. Then   the following assertions hold:
\begin{enumerate}
\item [\rm (1)]
$M\in\mathcal{I}_{K_{d-1}}^{(n,d)}(R)_{\leq k}$ if and only if $M\in\mathcal{A}_{K_{d-1}}(R)$ and $ K_{d-1}\otimes_{R}M \in\mathcal{I}^{(n,d)}(S)_{\leq k}$;
\item [\rm (2)]
 $N\in\mathcal{F}_{K_{d-1}}^{(n,d)}(S)_{\leq k}$ if and only if $N\in\mathcal{B}_{K_{d-1}}(S)$ and ${\rm Hom}_{S}(K_{d-1},N)\in\mathcal{F}^{(n,d)}(R)_{\leq k}$.
\end{enumerate}
\end{corollary} 
\begin{proof}
(1) $(\Longrightarrow)$ Let $M\in\mathcal{I}_{K_{d-1}}^{(n,d)}(R)_{\leq k}$. Then by Theorem \ref{2.r1}, $M\in\mathcal{A}_{ K_{d-1}}(R)$, and also by Proposition \ref{2.d1}(1), $ K_{d-1}\otimes_{R}M \in\mathcal{I}^{(n,d)}(S)_{\leq k}$.

$(\Longleftarrow)$ Let $M\in\mathcal{A}_{K_{d-1}}(R)$ and $ K_{d-1}\otimes_{R}M \in\mathcal{I}^{(n,d)}(S)_{\leq k}$.  Then it follows that  ${\rm Hom}_S(K_{d-1},K_{d-1}\otimes_{R}M)\cong M$, and also, there is an exact sequence 
$$0\longrightarrow K_{d-1}\otimes_{R}M\longrightarrow I_{0} \longrightarrow I_{1}\longrightarrow\cdots\longrightarrow I_{k}\longrightarrow0,$$
where any $I_{i}\in\mathcal{I}^{(n,d)}(S)$. So, there exists the following exact sequence of $R$-modules:
$$0\longrightarrow M\longrightarrow {\rm Hom}_{S}(K_{d-1},I_{0}) \longrightarrow\cdots\longrightarrow {\rm Hom}_{S}(K_{d-1},I_{k})\longrightarrow0,$$
where every ${\rm Hom}_{S}(K_{d-1},I_{i})\in\mathcal{I}_{K_{d-1}}^{(n,d)}(R)$, and then $M\in\mathcal{I}_{K_{d-1}}^{(n,d)}(R)_{\leq k}$.

(2) $(\Longrightarrow)$ Let $N\in\mathcal{F}_{K_{d-1}}^{(n,d)}(S)_{\leq k}$. Then by Proposition \ref{1.g1}(2), $N^{*}\in\mathcal{I}_{K_{d-1}}^{(n,d)}(S^{op})_{\leq k}$. So by (1), $N^{*}\in\mathcal{A}_{K_{d-1}}(S^{op})$ and $ K_{d-1}\otimes_{S^{op}}N^{*} \in\mathcal{I}^{(n,d)}(R^{op})_{\leq k}$. By Proposition \ref{3-5-A-BB}(2), $N\in\mathcal{B}_{K_{d-1}}(S)$. Also, by \cite[Proposition 2.56]{Rot2}, we have $ K_{d-1}\otimes_{S^{op}}N^{*}\cong N^{*}\otimes_{S}K_{d-1}$, and by \cite[Lemma 3.55]{Rot2},  $N^{*}\otimes_{S}K_{d-1}\cong {\rm Hom}_{S}(K_{d-1},N)^{*}$. So ${\rm Hom}_{S}(K_{d-1},N)^{*}\in\mathcal{I}^{(n,d)}(R^{op})_{\leq k}$ and consequently ${\rm Hom}_{S}(K_{d-1},N)\in\mathcal{F}^{(n,d)}(R)_{\leq k}$  by \cite[Proposition 3.1]{TX} and Corollary \ref{1.l-1}(2).

$(\Longleftarrow)$ It follows from  \cite[Proposition 2.3]{TX} and  Propositions \ref{1.g1}(1) and \ref{3-5-A-BB}(1).
\end{proof}

\begin{proposition}\label{2.s1}
Let $K_{d-1}$ be a special faithfully semidualizing bimodule. Then the following equalities hold.
\begin{enumerate}
\item [\rm (1)]
$(n,d)$.${\rm id}_{S}(M)=K_{d-1}$-$(n,d)$.${\rm id}_{R}({\rm Hom}_{S}(K_{d-1},M))$ for any $S$-module $M$;
\item [\rm (2)]
 $(n,d)$.${\rm fd}_{R}(M)=K_{d-1}$-$(n,d)$.${\rm fd}_{S}(K_{d-1}\otimes_{R}M)$ for any $R$-module $M$;
 \item [\rm (3)]
$K_{d-1}$-$(n,d)$.${\rm fd}_{S}(M)=(n,d)$.${\rm fd}_{R}({\rm Hom}_{S}(K_{d-1},M))$ for any $S$-module $M$;
\item [\rm (4)]
 $K_{d-1}$-$(n,d)$.${\rm id}_{R}(M)=(n,d)$.${\rm id}_{S}(K_{d-1}\otimes_{R}M)$ for any $R$-module $M$.
 \end{enumerate}
\end{proposition}
\begin{proof}
(2) Suppose that $(n,d)$.${\rm fd}_{R}(M)=k<\infty$. Then by Theorem \ref{2.r1}, $M\in\mathcal{A}_{ K_{d-1}}(R)$, and so ${\rm Tor}^{R}_{i}(K_{d-1},M)=0$ for any $i\geq 1$. Also, there exists an exact  sequence of the form $$0\longrightarrow F_k\longrightarrow \cdots \longrightarrow F_1\longrightarrow F_0 \longrightarrow M\longrightarrow0,$$ where any $F_j\in\mathcal{F}^{(n,d)}(R)$ for $0\leq j\leq k$. So we have the following exact sequence:
$$0\longrightarrow K_{d-1}\otimes_{R}F_k\longrightarrow \cdots \longrightarrow  K_{d-1}\otimes_{R}F_0 \longrightarrow K_{d-1}\otimes_{R}M\longrightarrow0,$$
where any  $ K_{d-1}\otimes_{R}F_j \in\mathcal{F}_{K_{d-1}}^{(n,d)}(S)$ and so $K_{d-1}$-$(n,d)$.${\rm fd}_{S}(K_{d-1}\otimes_{R}M)\leq k$.

Conversely, If $K_{d-1}$-$(n,d)$.${\rm fd}_{S}(K_{d-1}\otimes_{R}M)=k<\infty$, then by Proposition \ref{2.t1}(2), $ K_{d-1}\otimes_{R}M \in\mathcal{B}_{ K_{d-1}}(S)$. Hence by replacing $K_{d-1}$ instead $_SC_R$ from \cite[Lemma 2.9]{Z.TT}, we deduce that $M\in\mathcal{A}_{ K_{d-1}}(R)$, and consequently, we have isomorphism
$M\cong {\rm Hom}_{R}(K_{d-1},K_{d-1}\otimes_{R}M)$. Also,  there exists an exact
sequence $$\mathcal{X}=0\longrightarrow K_{d-1}\otimes_{R}F_k\longrightarrow \cdots \longrightarrow K_{d-1}\otimes_{R}F_1\longrightarrow K_{d-1}\otimes_{R}F_0 \longrightarrow K_{d-1}\otimes_{R}M\longrightarrow0,$$
of $S$-modules, where  $F_{j}\in\mathcal{F}^{(n,d)}(R)$ for any  $0\leq j\leq k$. On the other hand, by Proposition \ref{1.e1}(2), we have $ K_{d-1}\otimes_{R}F_j \in\mathcal{B}_{ K_{d-1}}(S)$, since $ K_{d-1}\otimes_{R}F_j \in\mathcal{F}_{K_{d-1}}^{(n,d)}(S)$. Therefore by  Definition of Bass,   for any $i\geq 1$ we have 
$${\rm Ext}^{i}_{S}(K_{d-1},K_{d-1}\otimes_{R}F_j)=0\ \ \ , \ \ \ {\rm Ext}^{i}_{S}(K_{d-1},K_{d-1}\otimes_{R}M)=0,$$
and hence   ${\rm Hom}_{S}(K_{d-1},\mathcal{X})$ is exact. On the other hand, $F_j\in\mathcal{A}_{ K_{d-1}}(R)$ by Theorem \ref{1.b1}(2). So  $F_j\cong {\rm Hom}_{R}(K_{d-1},K_{d-1}\otimes_{R}F_j)$.
Hence, there is the following commutative
diagram with the lower row exact:
$$\xymatrix{
0\ar[r]&F_{k}\ar[r]\ar[d]^{\cong}&\cdots\ar[d]^{\cong} \ar[r]&M \ar[r]\ar[d]^{\cong}&0  \\
0\ar[r]&{\rm Hom}_{R}(K_{d-1},K_{d-1}\otimes_{R}F_{k})\ar[r]&\cdots \ar[r]&{\rm Hom}_{R}(K_{d-1},K_{d-1}\otimes_{R}M)\ar[r]&0, \\
}$$
where the upper row is an exact sequence of $R$-modules and any $F_j\in\mathcal{F}^{(n,d)}(R)$ for $0\leq j\leq k$. Then we obtain that $(n,d)$.${\rm fd}_{R}(M)\leq k$. Similarly, cases (1), (3) and (4) are follow.
\end{proof}
\begin{proposition}\label{1.8}
The following statements hold.
\begin{enumerate}
\item [\rm (1)]
If $S$ is an $n$-coherent ring, then the class  $\mathcal{I}^{(n,d)}(S)_{<\infty}$ is closed under extentions, kernels of epimorphisms and cokernels of monomorphisms;
\item [\rm (2)]
If $R$ is an $n$-coherent ring, then the class $\mathcal{F}^{(n,d)}(R)_{<\infty}$ is closed under extentions, kernels of epimorphisms and cokernels of monomorphisms.
\end{enumerate}
\end{proposition}
\begin{proof}
Let  $0\rightarrow M^{'}\rightarrow M\rightarrow M^{''}\rightarrow 0$ be an  exact sequence of $S$-modules. If $(n,d)$.${\rm id}_{S}(M^{'})\leq (n,d)$.${\rm id}_{S}(M^{''})\leq k<\infty$, then 
 there exist the exact sequences 
$$0\longrightarrow M^{'}\longrightarrow I^{'}_{0} \longrightarrow I^{'}_{1} \longrightarrow\cdots\longrightarrow I^{'}_{k-1}\longrightarrow D^{'}_k\longrightarrow0$$ and 
$$0\longrightarrow M^{''}\longrightarrow I^{''}_{0} \longrightarrow I^{''}_{1} \longrightarrow\cdots\longrightarrow I^{''}_{k-1}\longrightarrow D^{''}_k\longrightarrow0 $$
of $S$-modules, where each $I^{'}_{i}$ and $I^{''}_{i}$ are injective. Since $S$ is $n$-coherent, $0={\rm Ext}_S^{d+k+1}(U,M^{'})\cong{\rm Ext}_S^{d+1}(U,D^{'}_k) $ and also, 
$0={\rm Ext}_S^{d+k+1}(U,M^{''})\cong{\rm Ext}_S^{d+1}(U,D^{''}_k)$ for every finitely $n$-presented $S$-module $U$,  and so $D^{'}_k$ and $D^{''}_k$ are in $\mathcal{I}^{(n,d)}(S)$. So
 by horseshoe lemma, there exist the following exact sequences:
$$0\longrightarrow M\longrightarrow I^{'}_{0}\oplus I^{''}_{0} \longrightarrow I^{'}_{1}\oplus I^{'''}_{1}  \longrightarrow\cdots\longrightarrow I^{'}_{k-1}\oplus I^{'''}_{k-1} \longrightarrow D_k\longrightarrow0$$ 
 $$0\rightarrow D^{'}_k\rightarrow D_k\rightarrow D^{''}_k\rightarrow 0.$$ We easily get that $D_k\in \mathcal{I}^{(n,d)}(S)$, and so $(n,d)$.${\rm id}_{S}(M)\leq k.$
 
 If $(n,d)$.${\rm id}_{S}(M^{'})\leq (n,d)$.${\rm id}_{S}(M)\leq k<\infty$, then 
 there exist the exact sequences 
$$\mathcal{Y}_{1}=0\longrightarrow M^{'}\longrightarrow I^{'}_{0} \longrightarrow I^{'}_{1} \longrightarrow\cdots\longrightarrow I^{'}_{k-1}\longrightarrow I^{'}_k\longrightarrow0$$ 
$$\mathcal{Y}_{2}=0\longrightarrow M\longrightarrow I_{0} \longrightarrow I_{1} \longrightarrow\cdots\longrightarrow I_{k-1}\longrightarrow I_k\longrightarrow0 $$
of $S$-modules, where each $I^{'}_{i}$ and $I_{i}$ are in $\mathcal{I}^{(n,d)}(S)$.  By \cite[Theorem 2.20]{ZX}, every $S$-module has an $(n,d)$-injective preenvelope. So ${\rm Hom}_S(\mathcal{Y}_{1}, \mathcal{I}^{(n,d)}(S))$ and  ${\rm Hom}_S(\mathcal{Y}_{2}, \mathcal{I}^{(n,d)}(S))$ are exact, and then by \cite[Theorem 3.4]{HUU}, 
 there exist the exact sequences 
$$0\longrightarrow M^{''}\longrightarrow I \longrightarrow I_{1}\oplus I^{'}_{2}  \longrightarrow\cdots\longrightarrow I_{k-1}\oplus I^{'}_{k} \longrightarrow 0 $$ 
 $$0\rightarrow I^{'}_{0}\rightarrow I_{0}\oplus I^{'}_{1}\rightarrow I\rightarrow 0,$$ where by \cite[Proposition 3.1]{TX}, $I_{i}\oplus I^{'}_{j}$ is in $\mathcal{I}^{(n,d)}(S)$. Also, by \cite[Lemma 2.8]{TX}, $I\in\mathcal{I}^{(n,d)}(S)$, since every $(n,d)$-injective is $(n,d+1)$-injective. Consequently, we get that $(n,d)$.${\rm id}_{S}(M^{''})\leq k$.

 If $(n,d)$.${\rm id}_{S}(M^{''})\leq (n,d)$.${\rm id}_{S}(M)\leq k<\infty$, then 
 there exist the exact sequences 
$$\mathcal{X}_{1}=0\longrightarrow M^{''}\longrightarrow I^{''}_{0} \longrightarrow I^{''}_{1} \longrightarrow\cdots\longrightarrow I^{''}_{k-1}\longrightarrow I^{''}_k\longrightarrow0$$ 
$$\mathcal{X}_{2}=0\longrightarrow M\longrightarrow I_{0} \longrightarrow I_{1} \longrightarrow\cdots\longrightarrow I_{k-1}\longrightarrow I_k\longrightarrow0 $$
of $S$-modules, where each $I^{''}_{i}$ and $I_{i}$ are in $\mathcal{I}^{(n,d)}(S)$.  By \cite[Theorem 2.20]{ZX},  ${\rm Hom}_S(\mathcal{X}_{1}, \mathcal{I}^{(n,d)}(S))$ and  ${\rm Hom}_S(\mathcal{X}_{2}, \mathcal{I}^{(n,d)}(S))$ are exact, and then by \cite[Theorem 3.8]{HUU}, 
 there exist the exact sequence
$$0\longrightarrow M^{'}\longrightarrow I_0 \longrightarrow I^{''}_{0} \oplus I_{1}  \longrightarrow\cdots\longrightarrow I^{''}_{k-1}\oplus I_{k} \longrightarrow 0,$$ 
  where by \cite[Proposition 3.1]{TX}, $I_{i}\oplus I^{'}_{j}$ is in $\mathcal{I}^{(n,d)}(S)$, and so  $(n,d)$.${\rm id}_{S}(M^{'})\leq k$.
 
 (2) It is similar to the proof of (1) using of \cite[Theorems 3.2 and 3.6]{HUU} and \cite[Theorem 2.20]{ZX}.
\end{proof}

\begin{theorem}\label{1.9}
Let $K_{d-1}$ be a  special faithfully semidualizing bimodule. Then the following statements hold.
\begin{enumerate}
\item [\rm (1)]
If $S$ is an $n$-coherent ring, then the class  $\mathcal{I}_{K_{d-1}}^{(n,d)}(R)_{<\infty}$  is closed under extentions, kernels of epimorphisms and cokernels of monomorphisms;
\item [\rm (2)]
If $R$ is an $n$-coherent ring, then the class  $\mathcal{F}_{K_{d-1}}^{(n,d)}(S)_{<\infty}$  is closed under extentions, kernels of epimorphisms and cokernels of monomorphisms.
\end{enumerate}
\end{theorem}
\begin{proof}
 (1) Let  $0\rightarrow M^{'}\rightarrow M\rightarrow M^{''}\rightarrow 0$ be an  exact sequence of $R$-modules. If $K_{d-1}$-$(n,d)$.${\rm id}_{R}(M^{'})\leq K_{d-1}$-$(n,d)$.${\rm id}_{R}(M^{''})\leq k<\infty$, then  by Corollary \ref{1.ee1}(1),  $M^{'}, M^{''}\in\mathcal{A}_{K_{d-1}}(R)$.
  So by \cite[Corollary 6.3]{HW},  $M\in\mathcal{A}_{K_{d-1}}(R)$. Thus there is the following exact sequence:
$$0\longrightarrow K_{d-1}\otimes_{R}M^{'}\longrightarrow K_{d-1}\otimes_{R}M\longrightarrow K_{d-1}\otimes_{R}M^{''}\longrightarrow 0.$$
By Proposition \ref{2.s1}(4), $(n,d)$.${\rm id}_{S}(K_{d-1}\otimes_{R}M^{'}){\leq k}$ and $(n,d)$.${\rm id}_{S}(K_{d-1}\otimes_{R}M^{''}){\leq k}$. So by Proposition \ref{1.8}(1), $(n,d)$.${\rm id}_{S}(K_{d-1}\otimes_{R}M){\leq k}$,  and then by  Proposition \ref{2.s1}(4), $K_{d-1}$-$(n,d)$.${\rm id}_{R}(M)\leq k$.

If   max$\{K_{d-1}$-$(n,d)$.${\rm id}_{R}(M),  K_{d-1}$-$(n,d)$.${\rm id}_{R}(M^{'})\}\leq k<\infty$, then by Corollary \ref{1.ee1}(1),  $M, M^{'}\in\mathcal{A}_{K_{d-1}}(R)$.
  Hence by \cite[Corollary 6.3]{HW},  $M^{''}\in\mathcal{A}_{K_{d-1}}(R)$. So there is the following sequence:
$$0\longrightarrow K_{d-1}\otimes_{R}M^{'}\longrightarrow K_{d-1}\otimes_{R}M\longrightarrow K_{d-1}\otimes_{R}M^{''}\longrightarrow 0.$$
By Proposition \ref{2.s1}(4), $(n,d)$.${\rm id}_{S}(K_{d-1}\otimes_{R}M){\leq k}$ and $(n,d)$.${\rm id}_{S}(K_{d-1}\otimes_{R}M^{'}){\leq k}$. Then by Proposition \ref{1.8}(1), $(n,d)$.${\rm id}_{S}(K_{d-1}\otimes_{R}M^{''}){\leq k}$,  and so by  Proposition \ref{2.s1}(4), $K_{d-1}$-$(n,d)$.${\rm id}_{R}(M^{''})\leq k$. Similarly, we deduce that $\mathcal{I}^{(n,d)}_{K_{d-1}}(R)_{<\infty}$ is closed under  kernels of epimorphisms.

(2) It is similar to the proof of (1).
\end{proof}

 If $n=\infty$, then Theorem \ref{1.9} holds for any arbitrary ring. 
\begin{corollary}\label{1.900}
Let $K_{d-1}$ be a  special faithfully semidualizing bimodule. Then the following statements hold.
\begin{enumerate}
\item [\rm (1)]
The class  $\mathcal{I}_{K_{d-1}}^{(\infty,d)}(R)_{<\infty}$  is closed under extentions, kernels of epimorphisms and cokernels of monomorphisms;
\item [\rm (2)]
Tthe class  $\mathcal{F}_{K_{d-1}}^{(\infty,d)}(S)_{<\infty}$  is closed under extentions, kernels of epimorphisms and cokernels of monomorphisms.
\end{enumerate}
\end{corollary}

\bibliographystyle{amsplain}

\end{document}